\documentclass[11pt]{amsart}  
\usepackage{amssymb, latexsym, amsmath, amsfonts}
\usepackage{float}

\restylefloat{figure}
\usepackage{graphicx}

\newtheorem{thm}{Theorem}[section]
\newtheorem{cor}[thm]{Corollary}
\newtheorem{lem}[thm]{Lemma}
\newtheorem{prop}[thm]{Proposition}
\theoremstyle{definition}    
\newtheorem{defn}[thm]{Definition}
\theoremstyle{remark}
\newtheorem{rem}[thm]{Remark}  
\newtheorem{ex}[thm]{Example}
\numberwithin{equation}{section}

\newcommand{\al}{\alpha}

\begin{document}
\thanks{The author was supported by the DST INSPIRE Fellowship.}
\subjclass{Primary: 32A05; Secondary: 32A07}
\author{G. P. Balakumar}
\address{G.P. Balakumar: 
Indian Statistical Institute, Chennai -- 600113, India}
\email{gpbalakumar@gmail.com}
\pagestyle{plain}

\begin{abstract}
The purpose of this article is to provide an exposition of domains of convergence of power series of several complex variables without recourse to relatively advanced notions of convexity. 
\end{abstract} 
\title{Power series in several complex variables.}
\maketitle
\section{Notations, Preliminaries, Introduction.}
\noindent A nice exposition of a multidimensional analogue of the Cauchy  -- Hadamard formula on the radius of convergence of power series, can be found in the book \cite{S} by B. V. Shabat, which naturally leads one\footnote{This is being written with a graduate student in mind or those with no prior knowledge of the matter here. The remaining footnotes may be ignored on a first reading.} to the conviction that domains of convergence of a power series in several complex variables constitute precisely, the class of logarithmically convex complete multi-circular domains.  In the present expository essay, we provide an alternative route to this result which avoids relatively advanced notions of convexity, such as holomorphic convexity -- this is natural in a systematic presentation of the subject 
of several complex variables, where a first goal lies in obtaining various characterizations of the collective of {\it all} domains of holomorphy, of which domains of convergence of power series, form a very small (and the simplest) sub-class. 
We emphasize that this is an expository essay that has been inspired by Shabat's treatment \cite{S}. There have been other sources as well; instead of enlisting all the sources here, we shall cite them at appropriate places.\\

\noindent We show how one might guess the aforementioned result on the characterizing  features of domains of convergence of power series in higher dimensions and help develop a feel for this simplest class of domains of holomorphy. Indeed, we shall show that on any given logarithmically convex 
complete multi-circular domain $D\subset \mathbb{C}^N$, all power series with its domain of convergence coinciding with $D$, can be seen to arise in one particular fashion. Namely, every power series with $D$ as its domain of convergence, can be recast as a sum of monomials, indexed by sequences of rational points on the positive face of the standard simplex in $\mathbb{R}^N$, converging to prescribed points of a countable dense subset of the normalized effective domain of the support function of the logarithmic image of $D$! This then leads to a natural way of writing down explicit power series converging precisely on any such given $D$, without having to deal with the case of an unbounded $D$ separately as done in the nice set of lecture notes by H. Boas, available at his web-page \cite{Bo}. On the other hand given any power series, we shall see how to not only write down a defining function for the domain of convergence in $\mathbb{C}^N$ but also the support function of the convex domain in $\mathbb{R}^N$ formed by its logarithmic image, directly in terms of the coefficients of the given power series.  
All of this is perhaps folklore matter but our intent here is to provide a treatment from an elementary standpoint of our goal stated in the abstract, thoroughgoing on certain fundamental matters not found recorded or expounded upon in the literature to the knowledge of this author.\\

\noindent Let us set the stage up for our discussion to begin in the next section. Let $\mathbb{N}_0 = \mathbb{N} \cup \{0\}$ and $N \in \mathbb{N}$. For $J=(j_1, \ldots,j_N) \in \mathbb{N}_0^N$, define $\vert J \vert = \vert j_1 \vert + \ldots + \vert j_N \vert$
and for $z \in \mathbb{C}^N$, let $z^J$ stand for the monomial $z_1^{j_1}z_2^{j_2} \ldots z_N^{j_N}$. Let $(\mathbb{R}^+)^N$ denote the $N$-fold Cartesian product of the multiplicative group of $\mathbb{R}^+$ of positive reals; its closure in $\mathbb{R}^N$ is the monoid $(\mathbb{R}_+)^N$ 
with $\mathbb{R}_+$ being the multiplicative monoid of non-negative reals. For $J \in \mathbb{N}_0^N$, define
 $J!=j_1!j_2!\ldots j_N!$ with the understanding that $0!=1$.
We are interested here with the case $N>1$. Unless explicitly specified, our indexing set in all countable summations is $\mathbb{N}_0^N$. A connected open subset of $\mathbb{C}^N$ is called a {\it domain} \footnote{More generally, we shall refer to any connected open subset of any topological space $X$ as a domain in $X$.}.  A viewpoint which has been decisive for the exposition here, is that the most tangible manner of describing a domain is by supplying sufficient data about its boundary, the simplest of which is specifying a defining function for the boundary of the domain and when the domain is convex, the supporting function for it.  Two fundamental bounded domains which will appear often in the sequel are the unit ball with respect to the standard $l^2$-norm on $\mathbb{C}^N$ given by
\[
\mathbb{B}^N = \{ z \in \mathbb{C}^N \; : \; \vert z_1 \vert^2 + \ldots + \vert z_N \vert^2 < 1 \}
\]
and the unit ball with respect to the $l^\infty$-norm on $\mathbb{C}^N$ given by the $N$-fold Cartesian product of $\Delta$ the unit disc in $\mathbb{C}$, 
namely 
\[
\mathbb{U}^N := \Delta^N := \{ z \in \mathbb{C}^N \; : \; \vert z_j \vert <1 \text { for all } j=1,2, \ldots,N \},
\]
which is called the standard unit polydisc; while balls in the $l^\infty$-norm will be called polydiscs, balls in the $l^2$-norm will simply be referred to as `balls'. Further, $N$-fold Cartesian products of discs $\Delta(z^0_j,r_j)$ of varying radii $r_j$ and varying centers $z^0_j$ for $j$ varying 
through $\{1,2 ,\ldots, N\}$, called the polydisc with polyradius $r=(r_1, \ldots,r_N)$ centered at the point $z^0$ in $\mathbb{C}^N$, will be denoted by $P(z^0,r)$.  To indicate the practice of brevity in notation that will be adopted: the center of such sets will be dropped out of  notation and denoted $P$ or $B$, when it happens to be the origin or if they are not important for the discussion at hand; or for instance if the radius does need to be kept track of, discs in $\mathbb{C}$ about the origin with radius $r$ will be 
denoted $\Delta_{r}$. Finally, let us mention the one other norm to make an explicit appearance which is, the largest among all norms on $\mathbb{R}^N$ which assigns unit length to its standard basis vectors namely, the $l^1$-norm. Its unit ball is known by various names: co-cube/cross-polytope/orthoplex; the boundary of this orthoplex is the standard simplex $S_N$ and its intersection with the non-negative orthant is called the probability simplex given by
\[
PS_N = \{ x \in \mathbb{R}^N \; :\;  x_1  + \ldots +  x_N  = 1, \text{ and }  x_j  \geq 0 \text{ for all } j\}
\]
which may be noted to be the convex hull of the standard basis of $\mathbb{R}^N$.\\

\noindent  We summarize several basic facts that will be used tacitly in the sequel. Let $I$ be the unit interval $[0,1] \subset \mathbb{R}$, which may be noted to be closed under a pair of basic algebraic binary operations: one, the arithmetic mean and the other, the geometric mean of any two numbers from 
$[0,1]$. Infact, both these operations may be modified to give rise to a whole range of binary operations on $I$: for any pair of numbers $a,b$ their weighted arithmetic mean, corresponding to any 
fixed $t \in I$, is given by $(1-t)a + t b$ while their weighted geometric mean is given by $a^{1-t} b^{t}$. Furthermore, there is a relation between this pair of binary operations, given by the order relation, called the H\"older's inequality, namely,
\[
a^{1-t} b^{t} \leq (1-t)a + t b. 
\]   
The above family of binary operations \footnote{None of these binary operations of forming the means, is associative. The author thanks Prof. Harold Boas for pointing this out.} may be carried out on any sub-interval of $\mathbb{R}_+$ and coordinate-wise in higher dimensions as well, in an appropriate manner which we now discuss. Indeed, let $V$ be any finite dimensional vector space over the reals; there is for each $t \in I$ a pair of binary operations $\Phi_t, \Psi_t$. While one of them, to be the one denoted $\Phi_t$ in the sequel, 
 requires only the affine-space structure of $V$, the other requires coordinatizing $V$. Specifically, $\Phi_t$ corresponds to the action of forming the straight line segment joining \footnote{Straight line formation and convex sets can be defined in any affine space; circular arcs, to be introduced later, in affine spaces (of dimension at least two) with an origin i.e., vector spaces and logarithmic convexity in normed vector spaces.} a pair points; algebraically, $\Phi_t$ consists in forming coordinate-wise weighted arithmetic mean. Indeed, $\Phi_t: V \times V \to V$ is given by
\[
\Phi_t(v,w)= (1-t)v + tw.
\]
Subsets of $V$ closed under this binary operation for each $t \in I$ are the convex sets. As is apparent all the above-mentioned binary operations arise out of 
the basic pair of algebraic/arithmetic operations on the field of reals which themselves as such, keep playing a fundamental role. We pause for a moment to note that with the standard multiplication, $I$ is a monoid whose action on $V$ by $v \to tv$, is also of basic importance. Subsets 
of $V$ containing the origin and closed under this action, are the sets star-like with respect to the origin. 
Getting back from the digression now to the other binary operation, we first make some identification of $V$ with $\mathbb{R}^N$ for $N={\rm dim}(V)$; it is best defined first in the connected component of the identity of the multiplicative Lie group $(\mathbb{R}^*)^N$, namely $(\mathbb{R}^+)^N$, as:  
\[
\Psi_t(v,w) =  \Big(  \psi_t(v_1,w_1), \ldots, \psi_t(v_N,w_N)\Big),
\]
for $v,w \in (\mathbb{R}^+)^N$ with $\psi_t(a,b) = a^{1-t}b^t$. The former operation is facilitated by the scalar multiplication of $\mathbb{R}$ on $V$ (henceforth identified with $\mathbb{R}^N$) and the latter
\footnote{This binary operation which consists of forming the coordinate-wise geometric mean of the given pair of points, may be extended to all other cosets of $(\mathbb{R}^+)^N$ in $(\mathbb{C}^*)^N$ by taking coordinate-wise product with the map which sends a complex number $z$ to  $z/\vert z \vert$, as:
\[
\Psi_t(v,w) = \Big(  \vert v_1 \vert^{1-t} \vert w_1 \vert^t \frac{v_1w_1}{\vert v_1 w_1 \vert},\ldots, \vert v_N \vert^{1-t} \vert w_N \vert^t \frac{v_Nw_N}{\vert v_N w_N \vert}\Big)
\]
But we shall not pursue this here. We are more interested in sets closed under these binary operations -- which admit alternative definitions -- rather than the operations themselves.}
by its conjugate namely, the conjugate of scalar multiplication by the exponential/logarithm: 
\[
v \to \lambda^{-1}\big(t\lambda(v)\big)
\]
where $\lambda(v) = (\log v_1, \ldots, \log v_N)$. This logarithmic mapping $\lambda$ has an obvious extension: 
$ v \to  (\log \vert v_1 \vert, \ldots, \log \vert v_N \vert)$ as a surjective group homomorphism $(\mathbb{C}^*)^N \to \mathbb{R}^N$ whose kernel is
the torus $\mathbb{T}^N$. This map which we continue to denote by $\lambda$, may be further viewed to extend as a monoid morphism from the 
 multiplicative monoid \footnote{The multiplicative monoid structure on $\mathbb{C}^N$ is used in the operation $\Psi_t$ which plays a central role in this article:
$\Psi_t(v,w) = p\big( \lambda^{-1}\big(t\lambda(v)\big), \lambda^{-1}\big((1-t)\lambda(w)\big) \big)$ where $p(v,w) = (v_1w_1, \ldots, v_Nw_N)$ denotes the monoidal operation of coordinate-wise product. We remark in passing that the map $\lambda$ whose components may be thought of as $\Re\circ \log$ (for a suitable local branch of the complex logarithm) applied to the respective coordinates, is continuous, infact smooth and (pluri-)harmonic, on all of $(\mathbb{C}^*)^N$ even though the complex logarithm fails to be continuous on $\mathbb{C}^*$; if we factor out $\tau$ from $\lambda$, it is a local diffeomorphism, in particular, an open mapping. These facts are convenient in assuring ourselves, while imaging Reinhardt domains in the logarithmic space as {\it domains}. Finally, let us mention that its extension to $\mathbb{C}^N$ is upper semi-continuous; indeed $z \to \log \vert z \vert$ furnishes the simplest upper semicontinuous subharmonic function whose polar set is non-empty.} 
 $\mathbb{C}^N$ onto the additive monoid $[-\infty, \infty)^N$; this actually factors through the monoid morphism 
 $\tau: (z_1, \ldots,z_N)\to (\vert z_1 \vert, \ldots, \vert z_N \vert) $ mapping  $\mathbb{C}^N$ onto the absolute space $(\mathbb{R}_+)^N$. The product\footnote{Direct product of the additive monoid $[-\infty, \infty)^N$ with the multiplicative group $\mathbb{T}^N$ can be identified -- via 
the mapping $A:(x,\omega) \to (e^{x_1}\omega_1, \ldots, e^{x_N} \omega_N)$ --  with $\mathbb{C}^N$ which is an additive group as well as a multiplicative monoid.} of $[-\infty, \infty)^N$ with $\mathbb{T}^N$ can be identified via 
the mapping $A:(x,\omega) \to (e^{x_1}\omega_1, \ldots, e^{x_N} \omega_N)$ (here ofcourse it is understood that $e^{-\infty}=0$) with $\mathbb{C}^N$. Products of domains in $[-\infty, \infty)^N$ with $\mathbb{T}^N$ are pushed forward by this mapping $A$ onto domains which are `multi-circular'  (invariant under the natural action of $\mathbb{T}^N$ on $\mathbb{C}^N$) and are termed Reinhardt domains \footnote{It is helpful to draw (for $N\leq 3$) images of Reinhardt domains in the absolute space as well as in the corresponding logarithmic space and we urge the reader to do so.}. Pull-backs of convex domains in $\mathbb{R}^N$ by $\lambda$ are called logarithmically convex -- formulated again precisely in definition (\ref{logcvxdefn}) below. So, sets closed under $\Psi_t$ are those whose logarithmic images are closed under the former binary operation $\Phi_t$.
As $\Phi_t$ requires no coordinatization, it is trivial that sets closed under this binary operation for all $t \in I$ namely the convex sets, remain convex under all affine transformations -- convexity is an affine property. However, it is far more non-trivial that multi-circular logarithmically convex domains in $\mathbb{C}^N$ whose logarithmic images are complete/closed under translation by vectors from $(-\mathbb{R}^+)^N$, possess a property which remains invariant under all biholomorphic (not just affine!) transformations. This property known as pseudoconvexity will not be discussed much here (we refer the novice to Range's expository articles \cite{R1} and \cite{R2}). Pseudoconvexity is a subtle property; however, we hope that the present essay, among other extensive treatises such as \cite{JarPflu}, convinces the reader that it is possible to gain a `hands-on' experience with the simplest examples of `pseudoconvex' domains namely, domains of convergence of power series in several complex variables.\\

\noindent Among the most elementary functions of several complex variables are the monomial functions and their linear combinations. 
\begin{defn}
A function of the form $p(z) = \sum_{\vert J \vert\leq m} c_Jz^J$ is called a polynomial. Here, if at least one of the $c_J$'s with $\vert J \vert = m$ is non-zero, the total degree of $p$ is defined to be ${\rm deg}(p)=m$. For the zero polynomial, the degree is not defined. A polynomial is called homogeneous (of degree $m$) if the coefficients $c_J$ for $\vert J \vert <m$ are all zero. Equivalently, a polynomial $p$ of degree $m$ is homogeneous if and only if $p(\lambda z) = \lambda^m p(z)$ for all $\lambda \in \mathbb{C}$.
\end{defn}
\noindent Thus polynomials are for us by definition, functions on the coordinate space $\mathbb{C}^N$, defined by expressions from the 
(coordinate) ring $\mathbb{C}[z] = \mathbb{C}[z_1, \ldots,z_N]$. Such functions are annihilated by the operators $\partial/\partial \overline{z}_j$ for all $j=1,\ldots,N$
and are sometimes referred to as `holomorphic polynomials' to distinguish them from finite linear combinations of monomials in the $2N$ many `independent' variables $(z,\overline{z}) = (z_1, \ldots,z_N, \overline{z}_1, \ldots, \overline{z}_N)$. A basic question arising here is of the `independence' of $z$ from $\overline{z}$ which is addressed in basic complex analysis; for an advanced, enlightening treatment we refer the reader to \cite{dAn}. We shall not 
dwell anymore on this than   saying  that  each 
$z_j$ is annihilated by the operator $\partial/\partial \overline{z}_j$ for instance, where we request the reader to recall the notion of Wirtinger derivatives here: for example $\partial/\partial \overline{z}_j$ is the complex linear combination 
 of the standard partial differential operators 
$\partial/\partial x_j,\partial/\partial y_j$ given by $1/2 \; (\partial/\partial x_j + i \partial/\partial y_j)$
where $x_j = \Re z_j$, $y_j=\Im z_j$. Moving further, we may obtain more functions by taking limits of polynomials; but such limits will often not be well-defined on all of $\mathbb{C}^N$ and we need to identify the subset on which they exist. Before we investigate this, we must first be clear about issues of limits and convergence in several variables, which we review in the following sub-section.
\subsection{Series indexed by Lattices}
\noindent Suppose that for each $J \in \mathbb{N}_0^N$, a complex number $c_J$ is given; we may form the series $\sum c_J$ and discuss the matter of its convergence. A trouble immediately arising is: there is no canonical order on $\mathbb{N}_0^N$. So to start with, we make the following

\begin{defn} 
The series of complex numbers $\sum c_J$ indexed by $J\in \mathbb{N}_0^N$ is said to be convergent, if there exists at least one bijection 
$\phi: \mathbb{N} \to \mathbb{N}_0^N$ such that $\sum_{i=1}^{\infty} \vert c_{\phi(i)} \vert < \infty$. Then the number
\[
\sum_{i=1}^{\infty}  c_{\phi(i)} 
\]
is called the limit of the series. Now note that this notion of convergence is independent of the choice of the map $\phi$ and that it means absolute convergence, thus circumventing the ambiguities alluded to above; all possible rearranged-summing leads to the same sum.
\end{defn}

\begin{ex}[The geometric series of several variables.] 
\noindent  Let $r=(r_1, \ldots, r_N) \in \mathbb{R}_+^N$ with $r_i \in (0,1)$ for all $i=1, \ldots,N$. Then the number $r^J= r_1^{j_1} r_2^{j_2} \ldots r_N^{j_N}$ is again in $(0,1)$. If $I$ is a finite sub-lattice of $\mathbb{N}_0^N$, there is an integer $L$ such that $I \subset \{0,1,2, \ldots, L\}^L$ so that we may write
\[
\vert \sum\limits_{J \in I} r^J \vert = \sum_{J \in I} r^J \leq \prod\limits_{i=1}^N \sum\limits_{k_i=0}^L r_i^{k_i} \leq 
\prod\limits_{i=1}^N\Big( \frac{1}{1-r_i} \Big) < \infty
\]
and conclude that the series is convergent. Replacing $r$ by $z\in \Delta^N$ shows likewise that the partial sums of the multi-variable geometric series 
$\sum_J z^J$ is also convergent on $\Delta^N$ -- indeed, absolutely convergent -- with sum being given by 
$\prod\limits_{i=1}^N\big( 1/(1-z_i) \big)$.
\end{ex}

\subsection{Convergence of functions.} Let $M$ be an arbitrary subset of $\mathbb{C}^N$, $\{f_J \; : \; J \in \mathbb{N}_0^N\}$ a family of complex-valued functions on $M$. Denote by $\vert f_J \vert_M$ the supremum of $\vert f_ J \vert$ on $M$.
\begin{defn}
The series $\sum_J f_J$ is said to {\it normally convergent} on $M$ if the series of positive numbers $\sum \vert f_J \vert_M$ is convergent.
\end{defn}

\begin{prop}
Suppose the series $\sum f_J$ is normally convergent on $M$. Then it is convergent for any $z \in M$ and for any bijective map $\phi: \mathbb{N}_0 \to \mathbb{N}_0^N$, the series $\sum_{i=1}^{\infty} f_{\phi(i)}$ is uniformly convergent on $M$.
\end{prop}

\subsection*{Set theoretic operations.} A possibly not-so-often encountered operation shall arise naturally in the sequel, namely that of the limit infimum of a countable collection of sets, enumerated as say $\{C_n\}_{n=1}^{\infty}$; their limit infimum is given by
\[
\liminf\limits_{n \in \mathbb{N}} C_n = \bigcup\limits_{k=1}^{\infty} \bigcap\limits_{j=k}^{\infty} C_j. 
\]
Thus, $\omega \in \liminf_n C_n$ if and only if for some $n$, $\omega \in C_j$ for all $j \geq n$; in other words, $\omega \in \liminf_n C_n$
if and only if $\omega \in C_n$ eventually.
A trivial fact that will be useful to keep in mind for the sequel is that the limit infimum of a countable collection of {\it convex} sets in $\mathbb{R}^N$ is convex. Rudiments of convex analysis are reviewed in the last section which may be useful as a reference for our notational practices as well. Indeed it will do well to keep the basics of convex calculus afresh in mind and the basics for the present essay are summarized in the last section. 

\subsection{Recap of Convex Analysis and Geometry}
\noindent We shall relegate the recollection of fundamentals of convex analysis to the last section, except for the notion of {\it support function} which is so central to the sequel that we recall it here right away.
\begin{defn} \label{suppfn}
Let $C \subset \mathbb{R}^N$ be a closed convex set. The {\it support function} $h=h_C: \mathbb{R}^N \to (-\infty, +\infty]$ of $C$ is defined by
\[
h(u) = \sup \{ \langle x, u \rangle \; : \; x \in C \}.
\]
The set of all $u \in \mathbb{R}^N$, for which $h(u)$ is finite is called the {\it effective} domain of $h$ and we call its subset consisting unit vectors
thereof, as the {\it normalized effective} domain of $h$.
\end{defn}
\noindent For more, refer to the last section. 

\begin{rem}
A final remark about notations: an ambiguous notation to be used is the indexing of sequences of reals say, as $\{c^n\}$ rather than by a subscript, which may cause confusion with the notation of the $n$-th power of a number $c$. Such a notation will be employed only in connection with other objects; for instance, the first components of a vector sequence $v^n = (v^n_1, \ldots, v^n_N) \in \mathbb{R}^N$ is naturally denoted $v^n_1$. We hope such ambiguous notations will be clear from context.
\end{rem} 

{\it Acknowledgments:} The author would like to thank Kaushal Verma, Sivaguru Ravisankar and Harold Boas for suggesting improvements. 

\section{Power series in several variables.}
\begin{defn}
Let $c_J$ be a sequence of complex numbers indexed by $J \in \mathbb{N}_0^N$ and $z_0 \in \mathbb{C}^N$. Then the expression $\sum c_J (z-z_0)^J$ is called a formal power series about $z_0$. Without loss of generality, we shall assume henceforth that $z_0$ is the origin. If this series converges normally on a set $M$ to a complex-valued function $f$ then being a uniform limit of continuous functions, we first note that $f$ defines a continuous function on $M$.
\end{defn}
\begin{defn}
Let $f(z) = \sum c_J z^J$ be a formal power series. Denote by $B$ the set of all points of $\mathbb{C}^N$ at which the series $S$ converges; it's interior $B^0$ is termed the `domain' of convergence of the power series $S$. 
\end{defn}

\begin{rem}
\noindent 
\begin{enumerate}
\item[(i)] \noindent There is a canonical way to sum a power series of several variables, even though the indexing set in the summation is 
$\mathbb{N}_0^N$. Namely, one first sums up all monomials of any given degree and then sums up the homogeneous polynomials of various degrees thus obtained:
\[
\sum\limits_{k=1}^{\infty} \sum_{\vert J \vert=k} c_J z^J
\]
If we declare a power series to be convergent if the sum of its homogeneous constituents ordered by degree as above converges, instead of the (tacit)
requirement made above that every rearrangement of the constituting monomials of a power series must lead to a convergent series with the same sum, then the domain of convergence gets enlarged. As a power series is thought of more as a sum of the monomials constituting/occurring in it, this practice of summing by homogeneous components alone, is not adopted. It is even customary to write a power series as a sum of monomials arranged in non-decreasing order of their degree  (i.e., with respect to the partial order on $\mathbb{N}_0^N$ by $l^1$-norm) though our requirement places no emphasis on such an 
ordering.
\item[(ii)] \noindent We shall refer to both the formal power series and the (holomorphic) function it defines, by the same symbol.
\item[(iii)] \noindent  The quotes on the word `domain' in the above definition, can and will be dropped as soon as we verify that the $B^0$ is connected. This requires the following lemma.
\end{enumerate}
\end{rem}

\begin{lem} [Abel's lemma] Let $P' \Subset P$ be polydiscs about the origin i.e., $(P,P')$ is a pair of concentric (open) polydiscs with the closure of
 $P'$ being contained inside $P$. If the power series $\sum c_J z^J$ converges at some point of the distinguished boundary of $P$, then it converges normally on $P'$.\\
\noindent Here, the distinguished boundary of $P$ is the thin subset $\partial\Delta_{1} \times \partial\Delta_{2}\times  \ldots \times \partial\Delta_{N}$, of the boundary of $P$ if $P = \Delta_{1} \times \Delta_{2}\times \ldots \times \Delta_{N}$ where $\Delta_j$ are discs of some radii about the origin.
\end{lem}

\begin{proof}
Let $\partial_0 P$ denote the distinguished boundary of $P$. Let $w \in \partial_0 P$ be such that $\sum c_J w^J$ is convergent. Then firstly, there exists a constant $C>0$ such that $\vert c_Jw^J \vert \leq C$ for all $J \in \mathbb{N}_0^N$. Next, compare the modulii
of the coordinates of points in $P$ with that of $w$ i.e., consider the ratios $r_j(z) = \vert z_j \vert/\vert w_j \vert$ for $j=1,2, \ldots,N$ -- each of these ratios $r_j(z)$ is bounded above by a positive constant say $q_j$, strictly less than $1$, owing to $P'$ being compactly contained inside $P$.  
Note that the sup-norm of the monomial-function $c_J z^J$ on $P'$ is bounded above by the constant $Cq^J$:
\[
\vert c_J z^J \vert \leq \vert c_J q^Jw^J \vert \leq C q^J
\]
for every $z \in P'$. This comparison with the geometric series $\sum q^J$ -- which converges because we know $q_j$ are all strictly less than $1$ -- finishes the verification that $\sum \vert c_J z^J \vert_{P'}$ is convergent and subsequently that $\sum_J c_J z^J$ is normally convergent. Finally, since every compact subset of $P$ is contained in some compact sub-polydisc $P'$ of $P$, we see that our power series converges uniformly on each compact subset 
of $P$.
\end{proof}
\noindent We leave the following characterizing test to determine whether or not a point 
belongs to the interior of the set of convergence of a given power series, as an exercise. 
\begin{prop} \label{charactest}
A point $p$ belongs to the domain of convergence of a power series $\sum c_J z^J$ if and only if there exits a neighbourhood $U$ of $p$ and 
positive constants $M$ and $r<1$ such that
\[
\vert c_J z_1^{j_1} z_2^{j_2} \ldots z_N^{j_N} \vert \leq M r^{j_1 + \ldots +j_N}
\]
for all $J=(j_1, \ldots, j_N)\in \mathbb{N}^N$ and $z \in U$.
\end{prop}

\begin{defn}
We say that a power series $\sum_J c_J z^J$ converges compactly in a domain $D$, if it converges normally on every compact subset of $D$.
\end{defn}

\begin{lem} \label{Cauchyest}
Let $P' \Subset P$ be a pair of polydiscs about the origin. Suppose $f(z) = \sum c_J z^J$ converges compactly on the polydisc $P$ and the multi-radius of $P'$ is $r=(r_1, \ldots, r_N)$. Then the coefficients of the power series defining the function $f$, can be recovered from the knowledge of the values of $f$ on the distinguished boundary 
of $P'$ by the formula:
\[
c_K = \frac{1}{(2\pi)^N r^K} \int_{[0,2 \pi]^N} f(z) e^{-i(k_1 \theta_1 + \ldots +k_N \theta_N)} d \theta_1 \ldots d \theta_N
\]
and consequently, we have the estimate
\[
\vert c_K \vert \leq \frac{1}{(2 \pi)^N}\frac{\vert f \vert_{\mathbb{T}}}{r^K}
\]
\end{lem}

\begin{proof}
Set $z_j = r_j e^{i \theta_j}$ for each $j=1, \ldots, N$
to write
\[
f(z) = f(r_1e^{i \theta_1}, \ldots, r_N e^{i \theta_N}) = \sum c_J r^J e^{i(j_1 \theta_1 + \ldots +j_N \theta_N)}
\]
and integrate with respect to each of the variables $\theta_j$ on $[0, 2 \pi]$ to get
\begin{multline*}
\int_{[0,2 \pi]^N} f(z) e^{-i(k_1 \theta_1 + \ldots +k_N \theta_N)} d \theta_1 \ldots d \theta_N \\
= \sum c_J r^J \int_{[0,2 \pi]^N} e^{i(j_1 - k_1) \theta_1 + \ldots + (j_N - k_N)\theta_N} d \theta_1 \ldots d \theta_N .
\end{multline*}
where the interchange of integral and summation on the right is justified by the uniform convergence of our power series on the boundary of $P'$.
The integral appearing on the right in the last equation is zero except when $J=K$ in which case it is $(2 \pi)^N$. The formulae in assertion now follow.
\end{proof}

\begin{defn}
Let $z^0$ be any point of $(\mathbb{C}^*)^N$. The (open) polydisc centered at the origin with polyradius 
$(\vert z_1^0 \vert, \vert z_2^0 \vert, \ldots,\vert z_N^0 \vert)$ is called the polydisc spanned by the point $z^0$.
\end{defn}

\noindent We may rephrase Abel's lemma as follows. Let $P$ be a polydisc and $w$ a point of the distinguished boundary of $P$. If the power series 
$f(z)= \sum c_J z^J$ converges (unconditionally) at $w$, then  it converges compactly on $P$. Stated differently, if $f$ converges at a point $w$, then it converges compactly on the polydisc spanned by $w$. This means that the interior of the set of convergence of the general power series $f$ which we denoted $B^0$, can be expressed as the union of the (concentric) polydiscs spanned by points of $B$ and subsequently that $B^0$ must be connected. This finishes the pending verification that $B^0$ is indeed a domain. In fact, we may note more here: $B^0$ is what is known as a Reinhardt domain, indeed a `complete Reinhardt domain' as defined below and in particular therefore, a contractible domain.

\begin{defn}
A domain $D$ in $\mathbb{C}^N$ is termed Reinhardt (about the origin) if $ z \in D$ entails that $(e^{i \theta_1}z_1, \ldots, e^{i \theta_N}z_N) \in D$ for all possible choices of $(\theta_1,\ldots, \theta_N) \in \mathbb{R}^N$. Such a domain is also said to be {\it multi-circular}. A domain $D$ 
in $\mathbb{C}^N$ is said to be circular if $ z \in D$ entails (only) that $(e^{i \theta}z_1, \ldots, e^{i \theta}z_N) \in D$ for all $\theta \in \mathbb{R}$; it is said to be complete circular if it admits an action by the disc i.e., $ z \in D$ entails that $(\lambda z_1, \ldots, \lambda z_N) \in D$ for all $\lambda \in \overline{\Delta}$; complete circular domains are sometimes also referred to as complex star-like domains and we note in passing that being star-like with respect to the origin, all complete circular domains are contractible domains.
Likewise a Reinhardt domain is said to be {\it complete} if it is invariant under the action of the closed unit poydisc by coordinate-wise multiplication i.e., $z \in D$ entails that 
$(\lambda_1 z_1, \ldots, \lambda_N z_N) \in D$ for all choices of $(\lambda_1, \ldots,\lambda_N) \in \overline{\Delta}^N$.
\end{defn}

\begin{prop}
The domain of convergence $B^0$ of $f(z)$ is a complete Reinhardt domain and $f(z)$ converges compactly in $B^0$.
\end{prop}

\noindent Now we may ask: is every complete Reinhardt domain, the domain of convergence of some power series?
The answer is No: take a union $L$ of two concentric polydics of different polyradii about the origin in $\mathbb{C}^2$ for convenience say, whose absolute 
profile is as shown in the left hand part of figure \ref{fig:Logconvexity}.
\begin{figure}[h]
  \includegraphics[scale=0.4]{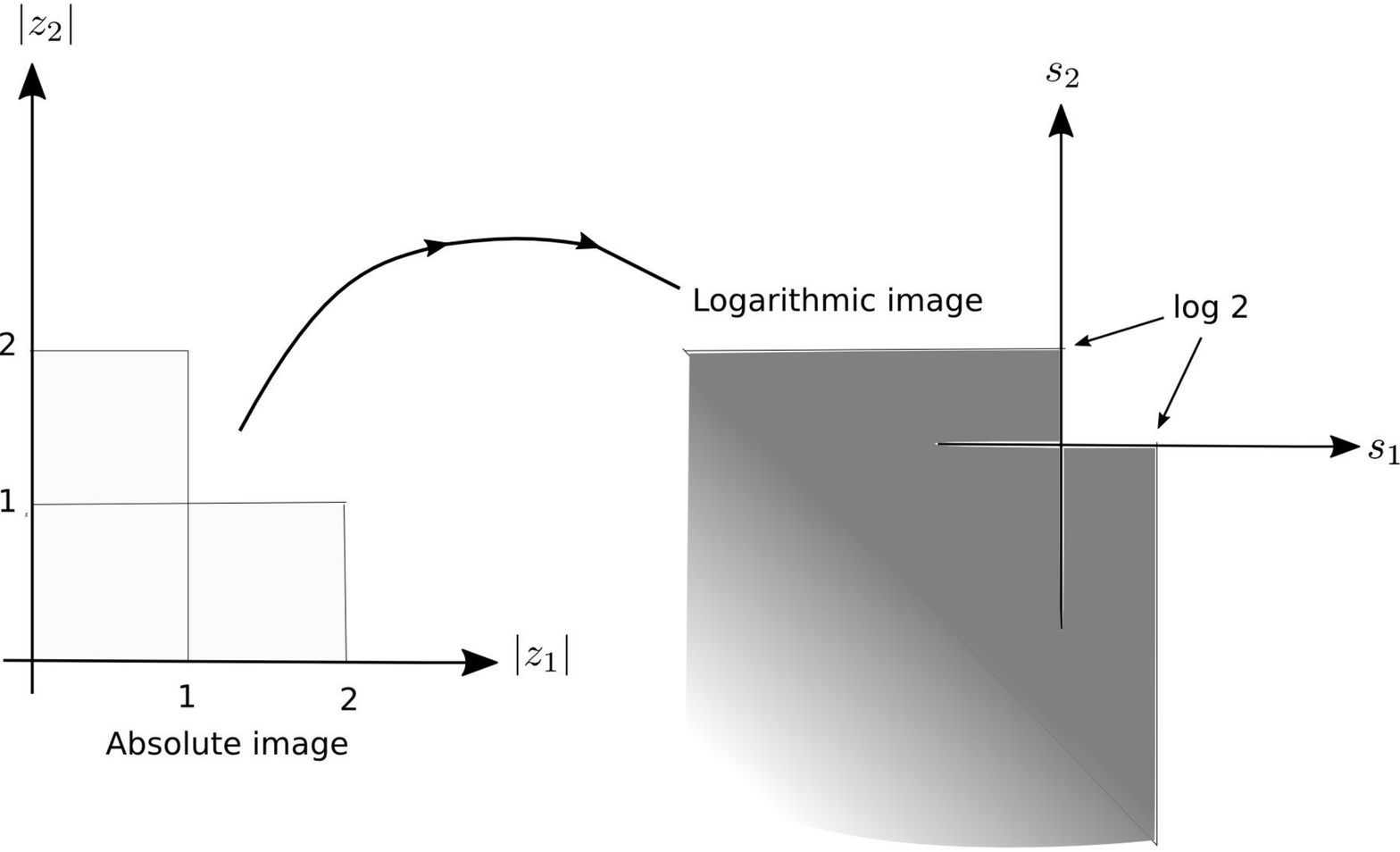}
  \caption{Profiles of the non-logarithmically convex, complete Reinhardt domain $L$.}
  \label{fig:Logconvexity}
\end{figure}

\noindent It is a fact that every power series convergent on this union
 actually converges on a larger domain; specifically, every power series which converges on $L$  converges on its `holomorphic hull', also indicated in figure \ref{fig:overconvergence} and given by 
 
\begin{figure}
  \includegraphics[scale=0.25]{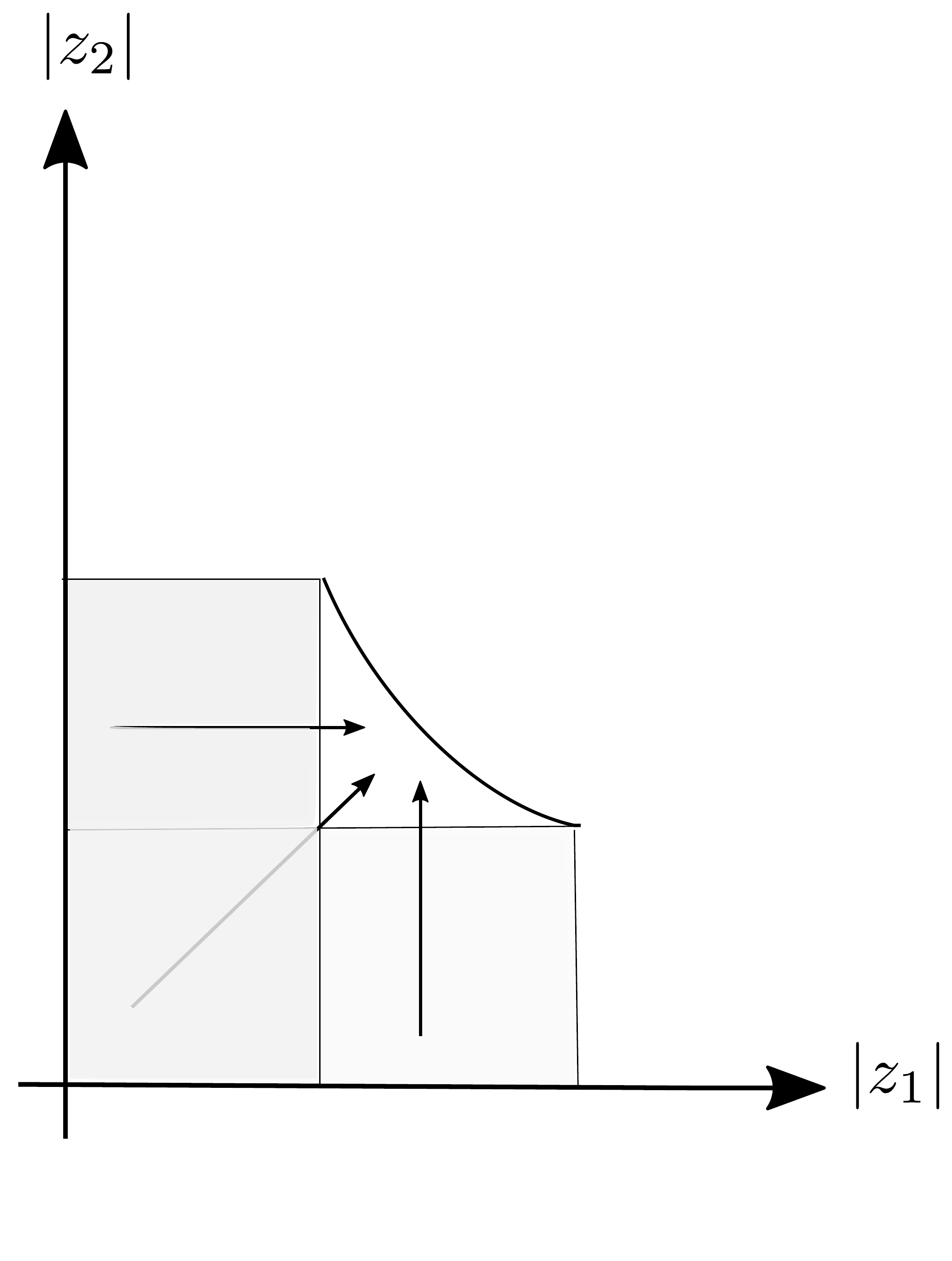}
  \caption{Absolute image of the holomorphic hull $\tilde{L}$ of the domain $L$.}
  \label{fig:overconvergence}
\end{figure}
 
\noindent 
\[
 \tilde{L} = \Delta^2\Big({\rm origin}; (2,2) \Big) \cap \{ (z,w) \in \mathbb{C}^2 \; :\; \vert z \vert \vert w \vert < 2 \}.
\]
This is only a `germ' of an instance of the Hartogs phenomenon peculiar to dimensions $n>1$; the reason for this compulsory (over-)convergence
in this particular example $L$, is actually no less simpler than for the general case of any complete Reinhardt domain which is not `logarithmically convex' ($L$ being a simple example). We therefore proceed directly towards the general case. Suffice it to say for now in short that, domains of convergence have some additional properties, logarithmic convexity being one which we now define. 

\begin{defn}\label{logcvxdefn}
Let $\lambda: (\mathbb{C}^*)^N \to \mathbb{R}^N$ be the map given by
\[
 \lambda(z) = (\log \vert z_1 \vert, \log \vert z_2 \vert,  \ldots,  \log \vert z_N \vert).
 \]
A Reinhardt domain $D$ in $\mathbb{C}^N$ is termed logarithmically convex if its logarithmic image $ \lambda(D^*)$, where 
\[
D^* = \{ z \in D \; :\; z_1 z_2 \ldots z_N \neq 0 \} = D \cap (\mathbb{C}^*)^N
\]
-- necessarily non-empty! -- is a convex set in $\mathbb{R}^N$. The set of points $z$, (atleast) one of whose coordinates is zero, forms the complex analytic variety
 $A = \{ z \in \mathbb{C}^N \; : \; z_1z_2 \ldots z_N = 0 \}$, which we shall  refer
to as the complex coordinate frame.
\end{defn}

\begin{rem}
\noindent For a logarithmically convex complete Reinhardt domain $D$, we shall sometimes write $\lambda(D)$ for $\lambda(D^*)$. We may also consider the map 
$\lambda : D \to [-\infty, \infty)^N$ with the obvious extension of $\lambda$ to points in $A \cap D$ i.e., points with some of its coordinates zero. 
Now, suppose $z^0 \in D \cap (\mathbb{C}^*)^N $. Denote the restriction of the mapping $\lambda$ to $\vert D^*\vert := \tau(D^*)$, by the symbol $\lambda/\tau$ and note that its Jacobian at the point $\tau(z^0)$ is given by 
\[
\frac{1}{\vert z^0_1 \vert} \ldots \frac{1}{\vert z^0_N \vert}
\]
which is evidently non-zero. Consequently by the inverse function theorem, $\lambda/\tau$ is an open map when restricted to $\vert D^*\vert $; it is not difficult to check that $\lambda/\tau$ (without restriction) is itself an open mapping. Further, it is easily seen that $\tau$ is also an open map.
Consequently, the composite of these maps $\lambda/\tau$ and $\tau$ namely, $\lambda$ itself, is an open mapping. So if $p,q \in D^*$ then $p_\lambda = \lambda(p)$, $q_\lambda=\lambda(q)$ are interior points of the convex domain $G = \lambda(D^*)$. Observe that 
every point of the line segment joining the pair $p_\lambda, q_\lambda$ is an interior point of $G$.
\end{rem}

\noindent Suppose $p,q$ are two different points in the domain of convergence of a given power series $\sum c_J z^J$. So, 
$\sum_J \vert c_J p^J \vert$ and $\sum_J \vert c_J q^J \vert$ converge to some finite positive numbers. Recall H\"older's inequality and write
it, as applied to the pair of positive numbers $\vert p^J \vert,\vert q^J \vert$ and the conjugate exponents $1/t, 1/(1-t)$ where $t \in (0,1)$, as follows:
\[
\vert p^J \vert ^t \vert q^J \vert^{1-t} \leq \frac{\big(\vert p^J \vert^t \big)^{1/t}}{1/t}  + \frac{\big(\vert q^J \vert^{1-t}\big)^{1/(1-t)} }{1/(1-t)}.
\]
Hence,
\[
\vert p^J \vert ^t \vert q^J \vert^{1-t} \leq t \vert p^J \vert + (1-t) \vert q^J \vert
\]
from which it is apparent that the given series converges at the point with real coordinates given by
\[
(\vert p_1 \vert^{t}\vert q_1 \vert^{1-t}, \vert p_2\vert^{t}\vert q_2 \vert^{1-t}, \ldots, \vert p_N \vert^{t} \vert q_N\vert^{1-t}).
\]
Infact, this point lies in the interior of the set of points where the power series converges, namely $D=B^0$. Indeed to sketch the reasoning here, suppose 
$p,q \in D^*=  D \cap (\mathbb{C}^*)^N $.  
Note that the logarithmic image of the above point lies on the line segment joining $\lambda(p),\lambda(q)$. Let $B_p, B_q$ be balls centered at the points $\lambda(p),\lambda(q)$ of radius some $\epsilon>0$, such that 
they are contained within $G= \lambda(D^*)$ and such that proposition \ref{charactest} holds therein i.e.,
there exists positive constants $r<1$ and $M_0$ such that 
\[
(j_1 \log \vert z_1 \vert + \ldots + j_N \log \vert z_N \vert)/\vert J \vert < M_0  \log r
\]
holds for all $z \in B_p \cup B_q$. The range of the validity of the above inequality then extends to the convex hull of $B_p$ with $B_q$ by the concavity of the logarithm. As $\lambda$ pulls back open sets to open sets just by continuity of $\lambda$, in particular the (open) convex hull of $B_p$ with $B_q$, in view of proposition \ref{charactest} again, the following basic result follows.
\begin{prop}
The domain of convergence $B^0$ is logarithmically convex.
\end{prop}

\noindent Well, how does one `discover' this? How can one guess other properties, if any, that is possessed by all those domains which are precise domains of convergence of some power series? Is it possible to pin down all common features shared by domains of convergence of power series which characterize them completely? To answer all this, one needs to get to the roots of the theory of power series: first, the (precise/largest) domain of convergence of any given power series in a single variable is always a disc whose radius is read off from the coefficients of the given series, using the following

\begin{thm}[Cauchy -- Hadamard formula]
The radius of convergence of the power series $\sum c_j z^j$ is given by
\[
\frac{1}{\limsup_{j \to \infty} \sqrt[j]{\vert c_j \vert}}
\]
\end{thm}

\noindent It is natural to ask for a constructive method of describing the domain of convergence of a power series of several variables. Now, the uniformity in the shape of the domain of convergence of power series in several variables is not as trivial as in the case of one variable, for, as we shall see, the ball and the polydisc are each, the natural domain of convergence of some power series but they are not biholomorphically equivalent. 
All we know at this point, is that domains of convergence of power series in several variables are also completely determined by their absolute profile, so we may focus on $\tau(D)$; but then $\tau(D)$ is not open and to avoid this annoyance, we pass to the logarithmic image $\lambda(D)$. More importantly, 
$\lambda(D)$ has a geometric property namely convexity, shared by all domains of convergence of power series. Further, they can be expressed as the union of concentric polydiscs.

\begin{defn}
A polydisc $U=U(z^0,r)$ is termed a polydisc of convergence of $\sum c_J z^J$ if $U \subset B$ but in any polydisc $U(z^0, R)$ where each $R_j \geq r_j$ for $j=1, 2, \ldots, N$ with at least one of the inequalities being strict, there are points in $U(z^0, R)$ where the series diverges.\\

\noindent  Every such polyradii $(r_1,r_2,\ldots,r_n)$ of $U(z^0,r)$ is called a conjugate polyradii i.e., the radii of each polydisc of convergence are called conjugate radii of convergence.
\end{defn}
\noindent If we join the dots formed by the various conjugate radii in the absolute space, what do we get? The answer to this is facilitated by a higher dimensional analogue of the Cauchy -- Hadamard formula:
\begin{prop} 
The conjugate radii of convergence of the power series $\sum\limits_{k=1}^{\infty}\sum_{\vert J \vert=k} c_J z^J$ satisfy the relation
\begin{equation}\label{R}
\limsup_{\vert J \vert \to \infty} \sqrt[\vert J \vert]{\vert c_J r^J \vert}=1
\end{equation}
i.e., $\limsup\limits_{k\to \infty} \max\{|c_J r^J|^{1/k}: |J|=k\} = 1$.
\end{prop}
\begin{proof}
Let $r$ be a conjugate radii of convergence of the given series
\begin{equation} \label{S}
\sum\limits_{k=1}^{\infty}\sum_{\vert J \vert=k} c_J z^J
\end{equation}
Let $\zeta \in \Delta$. Then $z= \zeta \cdot r$ lies in the polydisc of convergence $U$ of polyradius $r=(r_1,\ldots,r_N)$ about the origin, the series converges absolutely in $U$ and after regrouping the terms, we obtain from (\ref{S}), the following series in the variable $\zeta$:
\begin{equation*}
\sum\limits_{\vert J \vert=1}^{\infty} \vert c_J \vert z^J =  \sum\limits_{\vert J \vert=1}^{\infty} \vert c_J \vert  \zeta^{\vert J \vert} 
r^{\vert J \vert} = \sum\limits_{k=1}^{\infty}\Big(\sum_{\vert J \vert =k} \vert c_J \vert r^J\Big) \zeta^{\vert J \vert}
\end{equation*}
So we obtain from (\ref{S}) the series: 
\[
\sum\limits_{k=1}^{\infty}\Big(\sum_{\vert J \vert =k} \vert c_J \vert r^J\Big) \zeta^{\vert J \vert}
\]
which is a series in one complex variable $\zeta$, known to be convergent for $\zeta \in \Delta$.\\
\noindent If there exists $\zeta_0$ outside the closed unit disc at which this series converges, then it must be convergent on the disc centered at the origin of radius $\zeta_0$, which will imply that the coefficients satisfy the following decay estimate:
\[
\vert c_J \vert r^{\vert J \vert} \leq \frac{M}{\vert \zeta_0 \vert^{\vert J \vert}},
\]
that is $\vert c_J \vert \leq M/\vert (\zeta_0 r)^J \vert$.
This means that the series at (\ref{S}) must converge on the polydisc about the origin with polyradii $(\zeta_0 r_1, \zeta_0 r_2, \ldots, \zeta_0 r_N)$ contradicting that $U$ is a (maximal) polydisc of convergence. Thus, the series (\ref{S}) diverges for every point $\zeta$ with $\vert \zeta \vert>1$. By the Cauchy -- Hadamard formula for one variable, we therefore have
\begin{equation}\label{R'}
\limsup\limits_{k \to \infty} \sqrt[k]{\sum_{\vert J \vert = k} \vert c_J \vert r^J} =1.
\end{equation}
\noindent It only remains to show that this equation is equivalent to the one claimed in the statement of our proposition. For this, first choose among all the monomials $\{c_Jz^J\}$ with $\vert J \vert = j_1 + \ldots +j_n = k$, the one for which the maximum in 
\[
\max_{\vert J \vert= k} \vert c_J \vert r^J
\]
-- the maximum of sup-norms of monomials on the polydisc of radius $r$ -- is attained. Let $M=(m_1, m_2, \ldots,m_N)$ be such that this maximum is attained i.e.,
\[
\vert c_M \vert r^M  = \max_{\vert J \vert=k} \{ \vert c_ J \vert r^J\}.
\]
Then write down the obvious estimate
\[
\vert c_M \vert r^M \leq \sum_{\vert J \vert = k} \vert c_J \vert r^J \leq (k+1)^N \vert c_M \vert r^M,
\]
with the last inequality obtained by overestimating the number of terms appearing in the sum in the middle! Using this and the fact that $(k+1)^{N/K} \to 1$ as $K \to \infty$, we may rewrite (\ref{R'}) as the relation
\[
\limsup_{k \to \infty} \sqrt[k]{\vert c_M \vert r^M} = 1,
\]
from which the asserted relation of the proposition follows. 
\end{proof}
\noindent Now, note that the relation (\ref{R}) in the proposition above, can be rewritten as the equation
\begin{equation}\label{Rphi}
\varphi(r_1, r_2, \ldots,r_N)=0
\end{equation}
which `ties together' a relation among the conjugate radii of convergence of the series (\ref{S}). This equation determines the boundary of the domain 
$\tau(B^0)$ which depicts the domain of convergence $B^0$ in the absolute space. Next, substitute $r_j = e^{s_j}$ in (\ref{Rphi}). This leads to the last equation to be transformed as
\[
\psi(s_1, s_2, \ldots,s_N) = 0
\]
-- the equation for the boundary of $\lambda(B^0)$, the logarithmic image of $B^0$, some convex domain in $\mathbb{R}^N$. Indeed, let us rewrite equation (\ref{R}) after taking logarithms:
\[
\limsup\limits_{\vert J \vert \to \infty} \Big( \frac{j_1 \log r_1 + j_2 \log r_2 + \ldots + j_N \log r_N }{ j_1 + j_2 + \ldots +j_N} + \log \vert c_J \vert/\vert J \vert \Big) = 0
\]
So ultimately, in the variables $s_1, \ldots, s_N$, the relation (\ref{R}) reads:
\begin{equation}\label{Rpsi}
\limsup \limits_{\vert J \vert \to \infty} \Big( \frac{j_1 s_1 + j_2 s_2 + \ldots + j_N s_N }{ j_1 + j_2 + \ldots +j_N} + \log \vert c_J \vert/\vert J \vert \Big) = 0
\end{equation}
\noindent Indeed, the left hand side here is the function which we denoted by $\psi(s_1, \ldots,s_N)$ earlier; the above equation expresses $\psi$ as the limsup of a family, infact a sequence, of affine functions. Thus, $\psi$ must be convex. The domain of convergence $D$ of our given power series corresponds to the domain $G=\{ s \; : \; \psi(s) < 0 \}$. Let us rewrite this more precisely and record it for now: $D = \{ z \in \mathbb{C}^N \; : \; 
\varphi(z) <0\}$ where $\varphi$ is given in terms of the coefficients of our power series $\sum c_J z^J$ by
\begin{equation} \label{deffnforbdy}
 \varphi(z_1, \ldots, z_N) = \limsup_{\vert J \vert \to \infty} \sqrt[\vert J \vert]{\vert c_J z^J \vert} - 1.\\
\end{equation}
Thus, on the one hand, it {\it is} possible to read off the equation defining the boundary of its domain of convergence from its coefficients as in the one-variable case; on the other hand, as we shall see in what follows, the possibilities for the boundary is going to be as varied as a whole range of convex functions -- the mild restrictions to be satisfied by a convex function, in order for it to define the logarithmic image of the domain of convergence of some power series, can be found paraphrased at (\ref{deffncondn}).\\  

\noindent The radius of convergence in any direction specified by a unit vector $z \in \mathbb{C}^N$, is a function only of its radial component
$(\vert z_1 \vert, \ldots, \vert z_N \vert)$ which varies over the unit sphere in $\mathbb{R}^N$ and is given by
\[
R(z) = \frac{1}{\limsup_{\vert J \vert \to \infty} \sqrt[\vert J \vert]{\vert c_J z^J \vert}}.
\]
This is indeed the radial function of our domain of convergence $D$ which is star-like with respect to the origin. Recall that every star-like domain has associated to it a pair of special functions called the radial and (its reciprocal) the Minkowski gauge functional; for an immediate reference, the reader may consult the appendix material in the last section. So, every domain of convergence $D$ may also be described by its Minkowski gauge functional which turns out to be explicitly expressible in terms of the coefficients of any power series 
whose domain of convergence is $D$; as we find this elaborately dealt with care in \cite{JarPflu} (see lemma 1.5.13 therein), we shall minimize repetition and 
only note key formulae required for our purpose of describing certain salient features about the shape of the domains of convergence.
Let us express the aforementioned functions in terms of the defining function $\varphi$ for $D$ and the function $\psi$ defining its logarithmic image $G$ as above. The radial function reads:
\[
R(z) = \frac{1}{1+\varphi(z_1, \ldots, z_N)} = e^{-\psi(\log \vert z_1\vert, \ldots, \log \vert z_N \vert)}
\]
Consider the multitude of all power series with common domain of convergence $D$; it may be of some interest here to 
know if, this condition which holds, for all such series converging precisely on the same domain, can be cast analytically. 
To do so, we write out $\psi$ explicitly and the sought-for condition may be expressed as the following equation
for the gauge function $g=g_D$, directly in terms of the coefficients of {\it any} power series converging on our fixed $D$:
\[
g(z_1, \ldots, z_N) = \limsup_{\vert J \vert \to \infty} \big(\vert c_J \vert^{1/\vert J \vert} \vert z_1 \vert^{j_1/\vert J \vert} \ldots 
\vert z_N \vert^{j_N/\vert J \vert}\big).
\]
This when expressed in terms of $\psi$ reads: $g=\exp  \psi \circ \lambda$. This is saying that the function $g \circ {\rm Exp}$ is a logarithmically convex function, where ${\rm Exp}$ denotes the mapping
$\mathbb{R}^N \to (\mathbb{R}^+)^N$ given by ${\rm Exp}(s) = (e^{s_1}, \ldots, e^{s_N})$. 
As both the radial 
and gauge functions are functions of $(\vert z_1 \vert, \ldots, \vert z_N \vert)$, we shall by abuse of notation,
think of them as functions on $(\mathbb{R}_+)^N$ as well.
The convexity of $\psi$ and its finiteness on $\mathbb{R}^N$ (provided the domain of convergence is non-empty!) implies its continuity, so $-\log R(z) = \psi(\lambda(z))$ is continuous on $D^*$, thereby yielding the continuity of $g$ therein as well; this will play a crucial role in enlightening the topology of the boundaries of domains of convergence, to be discussed later. For now, note that 
\begin{equation} \label{-logR}
 - \log R\big({\rm Exp}(s_1, \ldots,s_N) \big) = \psi(s)
\end{equation}
where we already know $\psi$ to be a convex function. We conclude with the observation that the radial function of the domain 
of convergence of any power series has the property that $-\log R \circ {\rm Exp}$ 
is a convex function on $\mathbb{R}^N$. Finally, we remark in passing an analytic-cum-geometric characterization 
\footnote{This is a fundamental result in the subject of Several Complex Variables and is tantamount to characterizing which among such domains are `domains of holomorphy', to be briefed upon later in this article. We presume the reader would have an acquaintance with subharmonic functions; an upper semicontinuous function on a domain in $\mathbb{C}^N$ is termed plurisubharmonic if its restriction to each complex line intercepted by the domain is subharmonic.
Such characterizations of domains of holomorphy can be found in many books on the subject; the book by Vladimirov has been cited here, as it seems well-suited for study in parallel with this article.}
of those complete multicircular domains which qualify to be domains of convergence of some power series, given by the 
`plurisubharmonicity' of this fundamental function $-\log R$ where $R$ is the radial function of the given multicircular domain. 
These fundamental matters are well expounded in \cite{V} where such results are attained in the more general setting of 
Hartogs domains.\\

\noindent Before moving on, a bit of notation: let $p_K$ denote the point 
\[
(\frac{k_1}{k_1+ \ldots +k_N}, \ldots , \frac{k_N}{k_1+ \ldots +k_N})
\]
for $K \in \mathbb{N}^N$.   Let $PS{\mathbb{Q}^N}$ be the set of all such points $p_K$ which forms a countable dense subset of $P$. We note that $PS{\mathbb{Q}^N}$ is precisely the set of all points on $PS_N$ with rational coordinates.\\

\noindent We now proceed towards showing the existence for any given logarithmically convex complete multi-circular domain $D$ in $\mathbb{C}^N$, a power series whose domain of convergence is precisely $D$; we shall actually describe a method for writing down one explicitly. As the key property of $D$ is the convexity of the domain  
$G:=\lambda(D)$, we first study the link between the domain $G$ and the basic functions which constitute any power series namely, the monomial functions.
Notice first that monomial functions on 
$\mathbb{C}^N$ correspond to linear functionals on its logarithmic image.  More precisely, the monomial function
$z^J = z^{j_1} \ldots z^{j_N}$ transforms into the linear functional $ s \to j_1s_1 + \ldots j_N s_N$ on $\mathbb{R}^N$
whose kernel is therefore $H_J = \{s \in \mathbb{R}^N \; : \; j_1s_1 + \ldots j_N s_N = 0\}$. To spell out the result that we are after in brief, if an appropriate 
translate of $H_J$ is a supporting hyperplane  for $G$, then the (exponential of the) amount of translation required 
essentially renders the sought for coefficient of $z^J$ in our candidate power series provided, the norm of 
the gradient vector $(j_1, \ldots,j_N)$ 
is one -- we shall come to the appropriate choice of the norm in which we shall measure the amount of translation done, later.
To ensure this condition on the norm of the gradient is easy: we just need to divide out the defining equation for $H_J$ by 
$\vert J \vert$. But then notice that $z^J = z^{j_1} \ldots z^{j_N}$ with $J=(j_1, \ldots,j_N)= m (k_1, \ldots,k_N) = mK$ gives rise to the same 
$J/\vert J \vert$ as does $z^K=z^{k_1} \ldots z^{k_N}$. Our goal here, is to `discover' the above-mentioned result.\\

\noindent Recall our observation around equation (\ref{Rpsi}), that the logarithmic image $G_g$ of the {\it domain} of convergence of a given power series $g=\sum c_J z^J$ is the convex domain given by
\[
\Big\{ s \in \mathbb{R}^N \; : \; \limsup\limits_{\vert J \vert \to \infty} \Big(\frac{j_1 s_1 + j_2 s_2 + \ldots + j_N s_N }{ j_1 + j_2 + \ldots +j_N} + \frac{\log \vert c_J \vert}{\vert J \vert} \Big)  <0 \Big\}
\]
Observe that this is essentially equivalent to the statement that the logarithmic image of the domain of convergence of every power series is the liminf of a sequence of half-spaces whose gradient vectors belong to $PS\mathbb{Q}^N$. Indeed,
\begin{equation} \label{domcvgenrmlfrm}
G_g = \liminf_{J \in \mathbb{N}^N} \{ H_J \}
\end{equation}
with $H_J$ denotes the half-space $\{ s \; : \;  \langle J/\vert J \vert , s  \rangle + \log \vert c_J \vert^{1/\vert J \vert}  <0 \}$.\\

\noindent The fact that the gradients of the bounding/supporting hyperplanes for $G_g$ is `positive', is contained within the conditions imposed on our $D$. Indeed, continuing our study of the logarithmic image $G$, notice by the convexity of $G$ that any point $q \in \partial G$ has (possibly many) a supporting hyperplane for $G$ in $\mathbb{R}^N$ passing through it; let $H_q$ denote one such and be defined by say,
\[
A_q(x) := \langle m, x \rangle + c 
\]
where $m \in \mathbb{R}^n \setminus \{0\}$. So $A_q(q)=0$ and $A_q(x)$ is of the same sign throughout $G$. As usual, multiplying $A_q$ by $-1$ if necessary, we may assume $A_q$ is negative-valued throughout $G$. Just by the fact that $D$ has a neighbourhood of the origin contained in it, $G$ has a neighbourhood of $(-\infty, \ldots, -\infty)$ inside it; indeed, note that there is a positive number $M$ such that all points with  its coordinates all less than $-M$ must be contained in $G$ giving an infinite box-neighbourhood of $(-\infty, \ldots, -\infty)$ which is contained inside $G$ in its `left-bottom'. Further, the {\it complete} circularity of $D$ translates into the following condition about $G$: 
if $s^0 \in \overline{G}$ then all points $s$ with $s_j \leq s^0_j$ for all $j$, must also be contained in $\overline{G}$ -- this again gives an infinite box in the form of an orthant bounded by hyperplanes with gradients parallel to the axes, all passing through the point $s^0$. An illustration
is furnished in figure \ref{infinite-box}.
\begin{figure}
\includegraphics[scale=0.15]{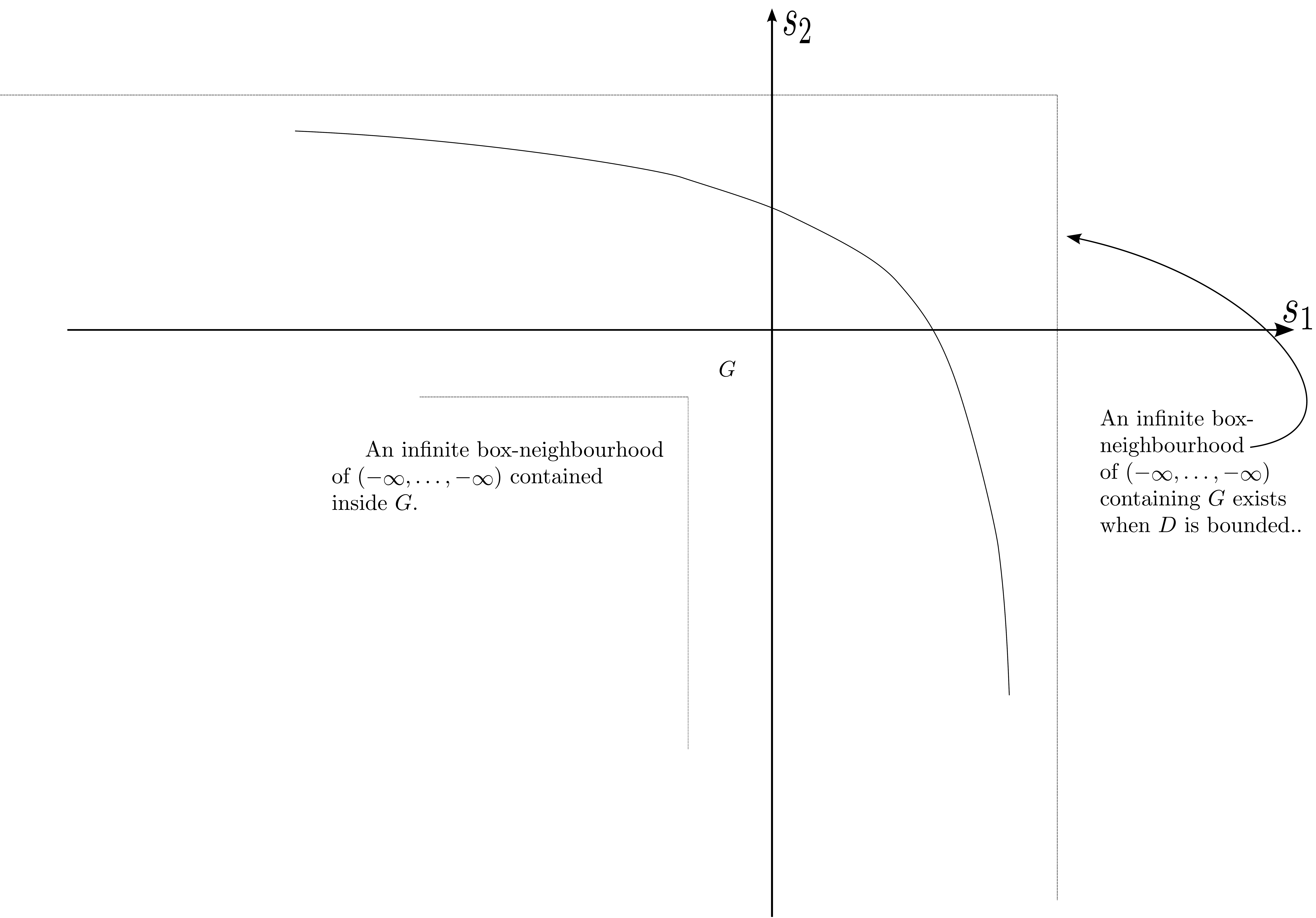}
 \caption{Logarithmic images of complete Reinhardt domains.}
\label{infinite-box}
\end{figure}
\noindent These features of $G$ force all the components $m_j$ of the gradient vector of $A_q$ to be non-negative; for if $m_j$ were negative for some $j$, then 
pick any $s^0 \in G$ and
consider points of the form 
\[
p(s) = (s_1^0, \ldots, s_{j-1}^0, s, s_{j+1}^0, \ldots, s_n^0 )
\]
with $s$, a negative number to be chosen soon. Then, on the one hand $p(s) \in G$ will imply 
\[
A_q(p(s)) = m_1s_1^0 + \ldots + m_{j-1}s_{j-1}^0 +m_js + m_{j+1}s_{j+1}^0 + \ldots + m_n s_n^0 <0, 
\]
which we rewrite as 
\[
m_js  < -\big(  m_1s_1^0 + \ldots + m_{j-1}s_{j-1}^0 + m_{j+1}s_{j+1}^0 + \ldots + m_n s_n^0\big).
\]
On the other hand, we can use the freedom to take $p(s)$ to be points in $G$ -- indeed, within the aforementioned infinite box-neighbourhood of $(-\infty, \ldots, -\infty)$ -- with $s$ negative and of modulus as large as we please; in particular, to contradict the above inequality whose right side is a constant. This shows that every component $m_j$ of the normal vector $m$ of every supporting hyperplane for $G$ must be non-negative.  
Hence, every supporting hyperplane for $G$, is given by an equation of the form $\{ s \in \mathbb{R}^N \; : \; A_q(s) =0 \}$ where
\[
A_q(s):= m_1 s_1 + \ldots + m_N s_n +d^q
\]
for some  positive non-negative numbers $m_j$, which needless to say, depend on $q$.  Actually, we may divide out the defining equation of this hyperplane by $\vert m_1 \vert + \ldots + \vert m_N \vert$ to assume that $m_j$'s are all numbers in $[0,1]$ with $\vert m_1 \vert + \ldots + \vert m_N \vert =1$ and we shall suppose so, in the sequel; this also results in a change in the constant $d^q$ but we shall continue to denote it by $d^q$. In other words $m$ lies in the non-negative face, denoted earlier by $PS_N$, of the standard simplex. This will be important in the sequel; so, let us spell this out explicitly here: the defining function for every supporting hyperplane for $G$ can be (and shall always be) written in a form such that its gradient vector belongs to $PS_N$. With this normalization made, $d^q$ in modulus, gives the distance of the hyperplane $H_q$ from the origin, as measured in the $l^\infty$-metric. Indeed, note first that
\[
\vert m_1 s_1 + \ldots +m_Ns_N \vert \leq \vert m \vert_{l^1} \vert s \vert_{l^\infty} = \vert s \vert_{l^\infty}.
\]
But then for $s \in H_q$, the left hand side is equal to $\vert d^q \vert$, which means that $\vert d^q \vert \leq \vert s \vert_{l^\infty}$ for all 
$s \in H_q$; noting that the point $s^q := (-d^q, \ldots, -d^q)$ satisfies $A_q(s^q)=0$ i.e., lies on  $H_q$ and 
has $\vert s^q \vert_{l^\infty}= \vert d^q \vert$, we get that the foregoing lower bound for the $l^\infty$-distance of points on $H_q$ from the origin 
is actually attained at the point $s^q$ and that this minimum distance is $\vert d^q \vert$. Let us keep these observations on record.\\

\noindent To discern the relationship between the coefficients defining a power series and the domain of convergence in more tangible terms, we now rephrase such relationships, (\ref{domcvgenrmlfrm}) being one such for instance, in terms of the support function rather than the defining function; while the defining function is a general tool to describe domains, the support function is a more convenient function specially adapted for convex domains. Let $g$ be some general power series $\sum c_J z^J$ with the logarithmic image of its domain of convergence $G_g$. Then, as we know $G_g=\{s \in \mathbb{R}^N
\; : \; \psi(s) <0\}$ with the defining function $\psi$ being given by
\begin{equation}\label{psidefn}
\psi(s)= \limsup\limits_{\vert J \vert \to \infty} \Big\{ \langle\frac{J}{\vert J \vert},s \rangle + \frac{1}{\vert J \vert} \log \vert c_J \vert \Big\}.
\end{equation}
Given any $\alpha \in PS_N$ pick any sequence $R^n= J^n/\vert J^n \vert \in PS\mathbb{Q}^N$ for some sequence $\{J^n\} \subset \mathbb{N}^N$, such that 
$R^n \to \alpha$ as $n \to \infty$. Then
\[
\psi(s) \geq \langle \alpha, s \rangle + \limsup\limits_{n \to \infty} \frac{\log \vert c_{J^n} \vert}{\vert J^n \vert }.
\]
Now, for all $s \in G_g$, $\psi(s) \leq 0$, so we must have
\[
\sup_{s \in G_g} \langle \alpha, s \rangle \leq - \limsup\limits_{n \to \infty}\big\{ \frac{\log \vert c_{J^n} \vert}{\vert J^n \vert } \big\}
= \liminf\limits_{n \to \infty} \{-\log \vert c_{J^n} \vert^{1/\vert J^n \vert}\}.
\]
This leads to the upper estimate for the support function  $h:= h_{G_g}$ of the convex domain $G_g$, given by
\begin{equation} \label{hupbd}
h(\alpha) \leq  - \limsup\limits_{n \to \infty} \frac{\log \vert c_{J^n} \vert}{\vert J^n \vert }
\end{equation}
with this being valid for all $\alpha \in PS_N$ and any sequence $\{J^n\} \subset \mathbb{N}^N$ with $J^n/\vert J^n \vert$ converging to $\alpha$ as $n \to \infty$. Stated differently, for every sequence $\{J^n\} \subset \mathbb{N}^N$ with $J^n / \vert J^n \vert$ being convergent to say $\alpha \in PS_N$, we have:
\begin{equation} \label{hupbdre}
-h(\alpha) \geq   \limsup\limits_{n \to \infty} \frac{\log \vert c_{J^n} \vert}{\vert J^n \vert }.
\end{equation} 
After passing to a subsequence to replace the limsup on the right by a limit, we may write
\[
 -h(\alpha) \geq \lim\limits_{n \to \infty} \frac{\log \vert c_{J^n} \vert}{\vert J^n \vert }.
\]
This means that every value assumed by $-h$ dominates some subsequential limit of $\log \vert c_J \vert/\vert J \vert$ leading us to the conclusion 
\begin{equation} \label{liminf}
\inf_{\alpha \in PS_N}\{ -h(\alpha)\} \geq  \liminf_{\vert J \vert \to \infty} \log \vert c_J \vert/\vert J \vert.
\end{equation}
\noindent Next suppose $\{K^n\} \subset \mathbb{N}^N$ is a sequence which achieves the limit supremum for the sequence $\log \vert c_J \vert/\vert J \vert$ i.e., 
$\log \vert c_{K^n} \vert/\vert K^n \vert$ is a convergent sequence with limit $\limsup(\log \vert c_J \vert/\vert J \vert)$. Then after passing to a 
subsequence of $\{K^n\}$ to assume $K^n/ \vert K^n \vert \to \gamma$ for some $\gamma \in PS_N$ and subsequently using (\ref{hupbdre}), we get
\begin{equation}
\limsup\limits_{\vert J \vert \to \infty} \log \vert c_J \vert/\vert J \vert \leq -h(\gamma) \leq \sup_{\alpha \in PS_N} \{ -h(\alpha)\}.
\end{equation}
On the other hand a lower bound may be obtained as follows. Pick {\it any} point  $s^0 \in \overline{G}_g$, recall (\ref{psidefn}) and write
\[
\limsup\limits_{\vert J \vert \to \infty} \langle \frac{J}{\vert J \vert}, s^0 \rangle + \limsup_{\vert J \vert \to \infty} 
\frac{\log \vert c_J \vert}{\vert J \vert} \geq \psi(s^0).
\]
As every subsequential limit of the countable collection of numbers $\{ \langle \frac{J}{\vert J \vert}, s^0 \rangle \; :\; J \in \mathbb{N}^N\}$ is of the form $\langle \alpha,s^0 \rangle$ for some $\al \in PS_N$, it follows that the left most term in the above, must be of the form
$\langle \beta, s^0 \rangle$ as well, for some $\beta \in PS_N$, so that we may write
\[
h(\beta) \geq \psi(s^0) - \limsup_{\vert J \vert \to \infty} \frac{\log \vert c_J \vert}{\vert J \vert}.
\]
As $s^0 \in \overline{G}_g$ was arbitrarily chosen, we may as well we might as well take $s^0$ to be on the boundary $\partial G_g$, to get the lower bound
\[
h(\beta) \geq - \limsup\limits_{\vert J \vert \to \infty} \frac{\log \vert c_J \vert}{\vert J \vert}.
\]
Now, rewrite this as:
\begin{equation} \label{exist}
\limsup\limits_{\vert J \vert \to \infty} \frac{\log \vert c_J \vert}{\vert J \vert} \geq -h(\beta)
\end{equation}
to subsequently derive from this, the lower bound:
\begin{equation} \label{lowerbd}
\limsup\limits_{\vert J \vert \to \infty} \frac{\log \vert c_J \vert}{\vert J \vert} \geq \inf_{\alpha \in PS_N} \{ -h(\alpha) \}.
\end{equation}
Now, (\ref{hupbdre}) and (\ref{exist}) together indicate the possibility that every value in the range of $-h$ can be realized as a subsequential limit of the sequence 
$\log \vert c_J \vert ^{1/ \vert J \vert}$. Indeed this is true: to this end, begin with the following rephrased version of (\ref{domcvgenrmlfrm}):
\[
G_g = \bigcap _{\alpha \in PS_N} \big\{ \langle \alpha, s \rangle +
 \limsup\limits_{ \{J^n\} \in S_{\alpha}} \frac{\log \vert c_{J^n} \vert}{\vert J^n \vert} <0 \big\}
 \]
where $S_{\alpha}$ is the set of all sequences $\{J^n\}$ in $\mathbb{N}^N$ with $J^n/\vert J^n \vert \to \alpha$. On the other hand, if $h$ is the support function of the convex domain $G_g$, we may write
\[
G_g = \bigcap_{\alpha \in PS_N} \big\{ \langle \alpha, s \rangle  - h(\alpha) <0 \big\}.
\]
Comparing the foregoing pair of representations of $G_g$, using the basic fact that for any convex domain, there can be at most one supporting hyperplane 
with a given gradient, we conclude that: for every $\alpha \in PS_N$,
\begin{equation}\label{Main}
h(\alpha) =  -\limsup\limits_{ \{J^n\} \in S_{\alpha}} \frac{\log \vert c_{J^n} \vert}{\vert J^n \vert}.
\end{equation}
Thus, just as we have a formula connecting the coefficients of a power series $g$ and the defining function of the logarithmic image $G_g$ of its domain of convergence, we have a similar one linking it to the support function of $G_g$, as well. By picking a suitable sequence $\{J^n\} \subset \mathbb{N}^N$ then, we may write
\begin{equation} \label{rangelimit}
h(\alpha) = \lim\limits_{n \to \infty}\big\{ -\log \vert c_{J^n} \vert^{1/\vert J^n \vert}\big\}
\end{equation}
where we are interested mainly in those $\alpha$ which lie in $PS_h = \{ \alpha \; : \; h(\alpha) \textrm{ is finite} \}$. In short, $h(\alpha)$ is a subsequential limit of $-\log \vert c_{J} \vert^{1/\vert J \vert }$, allowing us to finally conclude that the range of $h$ in $\mathbb{R}$ is contained in the set of all finite subsequential limits of the countable set of numbers:
\[
\big\{ -\log \vert c_{J} \vert^{1/\vert J \vert } \; : \; J \in \mathbb{N}^N \big\}.
\]
As every convex domain is characterized completely by its support function, it follows from (\ref{Main}) that: for any given convex domain $G$ with support function $h$, the coefficients of every power series $\sum c_J z^J$ which converges precisely on the domain $\lambda^{-1}(G)$, must satisfy (\ref{Main}) or equivalently
the following analogue of the Cauchy -- Hadamard formula for the radius of (the polydiscs of) convergence:
\begin{equation} \label{CHseveral}
e^{h(\alpha)} = \frac{1}{\limsup\limits_{\{J^n\} \in S_{\alpha}} \big\{ \vert c_{J^n} \vert^{1/\vert J^n \vert} \big\}},
\end{equation}
for each $\alpha \in PS_h$ (in fact, for all $\alpha \in PS_N$). Indeed, this formula gives the radius of convergence for any of the $\alpha$-constituents of our power series, where by an $\alpha$-constituent or $\alpha$-strand of our generic power series $\sum c_J z^J$ we mean any one of its sub-series given by
\[
\sum\limits_{n=1}^{\infty} c_{J^n} z_1^{J^n_1} \ldots z_N^{J^n_N}
\]
with $\{J^n\} \subset \mathbb{N}^N$ satisfying $J^n/ \vert J^n \vert \to \alpha$.\\

\noindent As logarithm is an increasing (=order-preserving) function, (\ref{CHseveral}) now leads to the result that the support function of the 
logarithmic image of the domain of convergence of any power series, can at least in
principle be completely determined from the coefficients through the formula:
\[
-h(\alpha) = \limsup\limits_{\{J^n\} \in S_{\alpha}} \big\{ \frac{\log\vert c_{J^n} \vert}{\vert J^n \vert} \big\}.
\]
As this holds for all $\alpha \in PS_N$, 
\[
 \limsup_{\vert J \vert \to \infty} \frac{ \log \vert c_J \vert}{\vert J \vert} 
\leq \sup_{\alpha \in PS_N} \{ -h(\alpha) \}
\]
Getting back now to (\ref{lowerbd}), we see that we have
\[
\inf_{\alpha \in PS_N}\{ -h(\alpha)\} \leq  \limsup_{\vert J \vert \to \infty} \frac{ \log \vert c_J \vert}{\vert J \vert} 
\leq \sup_{\alpha \in PS_N} \{ -h(\alpha) \}
\]
Combining this with (\ref{liminf}), we may therefore  write
\begin{equation}\label{interconlc}
\liminf_{\vert J \vert \to \infty} \frac{ \log \vert c_J \vert}{\vert J \vert}  \leq \inf_{\alpha \in PS_N}\{ -h(\alpha)\} \leq 
\limsup_{\vert J \vert \to \infty} \frac{\log \vert  c_J \vert}{\vert J \vert} \leq \sup_{\alpha \in PS_N} \{ -h(\alpha) \}.\\
\end{equation}
Now, recall our observation at (\ref{rangelimit}) that, every member in the range of $-h$ is actually a subsequential limit of $\log \vert c_J \vert/\vert J \vert$; this gives
\begin{align*}
\sup\limits_{\alpha \in PS_N} \{-h(\alpha)\}  &\leq \limsup\limits_{\vert J \vert \to \infty} \frac{\log \vert c_J \vert}{\vert J \vert} \text{ and }\\
\inf\limits_{\alpha \in PS_N} \{-h(\alpha)\}  &\geq \liminf\limits_{\vert J \vert \to \infty} \frac{\log \vert c_J \vert}{\vert J \vert} 
\end{align*}
While the second inequality here is one that we already know, the former when combined with (\ref{interconlc}), gives in conclusion:
\begin{align*}
\limsup\limits_{\vert J \vert \to \infty} \log \vert c_J \vert^{1/\vert J \vert} &= \sup_{\alpha \in PS_N} \{-h(\alpha)\}.
\end{align*}
As $h$ is a convex function, its range is an interval of the extended real line $\overline{\mathbb{R}}$. In conclusion, we therefore have that the   subsequential limits of $-\log \vert c_J \vert/\vert J \vert$ do not shoot above the range of the support function: $[\inf_{\alpha \in PS_N} \{h(\alpha)\} , \sup_{\alpha \in PS_N} \{h(\alpha)\}]$ while the set of all such limits contains this interval. We remark in passing that  since the range of $h$ may well be an infinite interval despite $G$ not being the whole space, it is 
 (\ref{CHseveral}) which will be more useful in practice.\\ 

\noindent Before proceeding to construct a power series which converges precisely on any given logarithmically convex multicircular domain $D$, let us take a look at two special cases: one when $D$ is the unit polydisc and another when $D$ is the pull-back of a half space under the map $\lambda$. For the former, the geometric series $\sum z^J$ which involves every monomial, converges precisely on the open unit polydisc (even though it {\it can} be analytically continued to a larger domain) whose support function is finite throughout $PS_N$. For the latter on the other hand, we may consider the power series
$\sum c_k z^{kJ}$ for some fixed $J \in \mathbb{N}^N$ whose domain of convergence has its logarithmic image $G_J$, determined as the limit infimum of half-spaces given by:
\begin{multline}
\liminf\limits_{k \to \infty}\Big\{ s \; : \; \frac{kj_1s_1+ \ldots + kj_Ns_N}{k(j_1+\ldots+j_N)} + \frac{\log \vert c_k \vert}{k(j_1+\ldots+j_N)}<0 \Big\}\\
=\liminf\limits_{k \to \infty}\Big\{ s \; : \; \frac{j_1s_1+ \ldots + j_Ns_N}{\vert J \vert} + \frac{1}{\vert J \vert}
\frac{\log \vert c_k \vert}{k}  <0 \Big\}\\
= \Big\{ s \in \mathbb{R}^N\; : \; \langle J, s \rangle + \limsup\limits_{k \to \infty} \log \vert c_k \vert^{1/k}    <0 \Big\}
\end{multline}
which is a single half-space obtained by translating the ortho-complement of $J$, by a distance $\limsup\limits_{k \to \infty} \log \vert c_k \vert/k$ in the direction opposite to $J$, unless the limsup in the above is infinite; if this limsup is $+\infty$ the half-space reduces to the null set whereas if this limsup happens to be $-\infty$, the domain of convergence becomes the whole space. The support function of  a half-space is finite precisely at a single point of $PS_N$ and for the above one, at $J /\vert J \vert \in PS\mathbb{Q}^N$. Indeed, if $h$ denotes the support function of the above half-space, then $h(J /\vert J \vert) = -1 /\vert J \vert \; \limsup\limits_{k \to \infty} \log \vert c_k \vert^{1/k}$.\\

\noindent Now, we may wish to write any general power series $\sum c_J z^J$ as a sum of series of the type just mentioned:
\[
g(z) = \sum\limits_{J \in \mathcal{P}}\Big( \sum\limits_{k=0}^{\infty} c_{kJ}z^{kJ} \Big)
\]
where $\mathcal{P}$ is the set of all $N$-tuples $J$ of positive integers whose greatest common divisor is one. This representation is supported by the absolute convergence of the power series on its domain of convergence $D$. Let $G_g = \lambda(D)$ denote the logarithmic image of $D$. As noted above, for each $J \in \mathcal{P}$ fixed, the logarithmic image $H_J^g$, of the domain of convergence of $\sum_{k=0}^{\infty} c_{kJ}z^{k J}$, is a half-space 
which may be {\it the whole space} $\mathbb{R}^N$ or empty as well. In fact, it may very well happen that every $H_J^g$ is the whole space $\mathbb{R}^N$, while $G_g$ is far from being so; this can be reconciled with the possibility that the set of points where the support function of $G_g$ is finite avoids all of the rational points of $PS_N$. On the other hand, if $H_J^g= \emptyset$ even for a single value of $J$, then 
$\limsup_{k \to \infty} \log \vert c_{kJ} \vert^{1/k}$ must be $+\infty$, which in turn means that the series representing $g(z)$ does not converge at any point $z$; so, $D = \emptyset$ as well. To address the more general question: how is the domain of convergence of $g$ related to these half-spaces $H_J^g$? briefly, 
suppose $\lambda(z) \in G_g$; then for all $J \in \mathcal{P}$ we have that
$\lambda(z)$ lies in the logarithmic image of the domain of convergence  of each $f_J$, of the form:
\[
f_J = \sum\limits_{k=0}^{\infty}c_{kJ}z^{kJ}.
\]
Thus, $\lambda(z)$ belongs to  $\cap H_J^g$ or in other words, 
\[
 G_g \subset \bigcap_{ R \in PS\mathbb{Q}^N} H_R^g.
\]
with
$H_R^g = \{s \in \mathbb{R}^N \; : \;  \langle R, s \rangle + \limsup_{k \to\infty} \log \vert c_{kJ} \vert^{1/\vert k \vert} <0\}$ where $R \in PS\mathbb{Q}^N$ is expressed as $J/\vert J \vert$.  We wish ofcourse to know whether this inclusion can be improved to a better estimate, first of all an equality. The foregoing set-theoretic upper bound on $G_g$, may be totally useless because this inclusion may be far from equality, for instance when $H_J^g = \mathbb{R}^N$ for all $J$, as mentioned above -- it is not difficult to conjure up examples when this takes place and in the forthcoming, we will see methods to do so; for now, consider for instance, the possibility of the domain of convergence of a power series $g$ of two complex variables, being such that its
logarithmic image in $\mathbb{R}^2$ is a half-space whose boundary is a line of `irrational slope' i. e., with gradient vector 
$(1,\alpha)$ for an irrational real number $\alpha$. On the one hand, we have the foregone equality
\[
G_g = \bigcap\limits_{\alpha \in PS_N}H_\alpha^g
\]
where
$H_\alpha^g = \{ s \in \mathbb{R}^N \; : \;  \langle \alpha, s \rangle - h(\alpha) <0\}$ with
$h$ being the {\it support function} of $G_g$. On the other hand, this equality does not immediately serve our purpose, as the intersection 
here is not countable; we shall redress this problem next -- what we are seeking here, is a procedure to cast any given power series $g$ as a sum of sub-series each with its
logarithmic image of its domain of convergence being a half-space and such that the intersection of these half-spaces yields $G_g$. Actually, 
it is enough if we can recover $G_g$ from the knowledge of these half-spaces by some tangible set-theoretic operation, not necessarily an intersection;
infact, the operation of limit infimum for sets, is the one which comes up in this context. The key point here is that while the indexing set for our half-spaces must be a countable (dense) collection of vectors from $PS_N$, it need not be $PS\mathbb{Q}^N$. Subsequently therefore, we shall shift our considerations a bit, to starting with arbitrary countable dense subsets of $PS_N$.  \\

\noindent Let $G$ be any convex domain in $\mathbb{R}^N$ with support function $h=h_G$. The effective domain of $h$ is the subset of those points of the domain of $h$ where the support function is finite. We shall refer to the subset of the effective domain $h$,
given by
\[
PS_h = \{ \alpha \in S_N \; : \; h(\alpha) \text{ is finite} \},
\]
as the {\it normalized domain}  (or normalized effective domain) of $h$, which is actually contained in $PS_N$, owing to the completeness of the given multicircular domain $D$, as noted earlier.
Let $\mathcal{C} =\{\alpha^n \}$ be an arbitrary countable dense subset of $PS_N$. Pick a sequence $\{J^{1k} \} \subset \mathbb{N}^N$ with 
\[
\Big( \frac{J^{1k}_1}{\vert J^{1k} \vert}, \ldots, \frac{J^{1k}_N}{\vert J^{1k} \vert} \Big) \to (\alpha^1_1, \ldots, \alpha^1_N),
\]
as $k \to \infty$ -- it is trivial to see that such a sequence exists. Next, pick a sequence $\{ J^{2k}\}$ this time in $\mathbb{N}^N \setminus \{J^{1k} \}$ such that
\[
\Big( \frac{J^{2k}_1}{\vert J^{2k} \vert}, \ldots, \frac{J^{2k}_N}{\vert J^{2k} \vert} \Big) \to (\alpha^2_1, \ldots, \alpha^2_N).
\]
Such a sequence exits, as $\pi\big( \mathbb{N}^N \setminus \{J^{1k} \} \big)$ is dense in $PS_N \setminus \{\alpha^1\}$, where 
$\pi(z)=z/\vert z \vert_{l^1}$. After $l$ steps, we would have sequences $\{J^{lk} \; : \; k \in \mathbb{N}^N\}$ such that for any $m \leq l$, we have 
\[
\{ J^{mk}\} \subset  \mathbb{N}^N \setminus \bigcup\limits_{i=1}^{m-1}\{J^{ik}:k \in \mathbb{N}\}
\]
and $J^{mk}/ \vert J^{mk} \vert \to \alpha^m$ as $k \to \infty$. Set $R^n_k = J^{nk}/ \vert J^{nk} \vert$. \\

\noindent Keeping the notations as in the foregoing para, let $g(z) = \sum c_J z^J$ be a power series with its domain of convergence $D$ and $\lambda(D)=G$ with support function $h$. We shall re-express the series $g$ as a sum indexed essentially by any chosen countable dense subset $\mathcal{C}$ drawn out of the normalized effective domain $PS_h$ of the support function. On the one hand, $PS_h$ may fail to have any rational points in particular, the support function may fail to be finite on integral points; on the other hand, the standard indexing of power series is through the standard positive integral lattice. In order to pass to the desired rearranged sum, we first set up approximating sequences for our chosen $\mathcal{C}$ drawn from $PS\mathbb{Q}^N$ as in the foregoing para. We may pick out a strand (=sub-series) of terms interspersed in $g$, corresponding to each such subsequence. Thereafter, 
look upon the series $g$, as an interlaced sum of such strands. More simply put, re-express $g$ in the following form 
\begin{equation}\label{recast}
\sum\limits_{n} \sum \limits_{k} c_{kn}z_1^{l_{kn}R^n_{k1}} z_2^{l_{kn}R^n_{k2}} \ldots z_N^{l_{kn}R^n_{kN}} + \text{ the remaining terms of }g
\end{equation}
with $l_{kn} = \vert J^{nk} \vert$ and $c_{kn} := c_{J^{nk}}$; the ordering of the `remaining terms' in the above, can be 
ignored by the absolute convergence of $g$ on $D$. We do not claim any uniqueness of the above expression of $g$ and infact, the `remaining terms' may be ignored altogether, because the values of the support function on the subset $PS_h$ (of $PS_N$) where it is finite, gets determined as follows: firstly, on the chosen countable dense subset $\mathcal{C}$ by the asymptotic behaviour of the coefficients 
of $g$ via the formula (\ref{CHseveral}):
\[
h(\alpha^n) = -\limsup\limits_{k \to \infty} \log \vert c_{kn} \vert^{1/\vert J^{nk} \vert},
\]
which subsequently, determines by continuity, the values of $h$ on all points of the relative interior of $PS_h$. As these values suffice to determine the 
convex domain $G$, this explains in what sense, we may ignore the `remaining terms', mentioned above. We have recast the power series $g$ as in (\ref{recast}) to peel-off information from various strands \footnote{A strand here means an infinite subset of the collection of coefficients; more precisely herein, one out of the infinitely many disjoint infinite subsets of the coefficients, each indexed by one of the sequences $\{J^{nk}: k \in \mathbb{N}\}$.} of coefficients of $g$ about the support function $h$ of its domain of convergence: (\ref{recast}) regroups $g$ a sum of its 
various $\alpha^n$-strands and it is this organization of its terms,  which splits up neatly to make apparent the links between the coefficients occurring in the various sections of the series $g$ and the geometry of its domain of convergence.
In conclusion, we thus observe here, how all power series arise `essentially' in the same manner: the `essential' limits being determined by  a convex domain in $\mathbb{R}^N$ through its support function and a countable dense subset of the
normalized domain of the support function.\\
 
\noindent A simple choice for getting a concrete/explicit power series converging precisely on a given log-convex Reinhardt $D$, now presents itself: take $c_{kn}$ such that $\vert c_{kn} \vert^{1/l_{kn}} = e^{-h(\alpha^n)}$. To substantiate a bit more explicitly why this surmise may work, we first observe that the problem of constructing a power series which converges precisely on the prescribed domain $D$, is equivalent to the geometric problem of expressing its logarithmic image $G=\lambda(D)$ as the limit infimum of a sequence of half-spaces whose bounding hyperplanes have their gradient vectors from $PS\mathbb{Q}^N$ and converge to a `dense' collection of supporting hyperplanes for 
the convex domain $G$. The gradient vectors of the supporting hyperplanes need not belong to $PS\mathbb{Q}^N$ at all; the foregoing prelude-para was to 
address this issue.  
So now, we choose a countable `dense' collection of supporting hyperplanes 
for the logarithmic image $G$ of our given domain, with the property that their (affine) defining functions all have 
gradient vectors whose components are all rational (and in $PS_N$); indeed, to be more careful and correct, make the choice such that the gradient vectors of the aforementioned  half-spaces,  are in the above notation, of the form $R_k^n = J^{nk}/l_{kn}$ 
where $l_{kn} = \vert J^{nk} \vert$ -- in particular therefore vectors from $PS\mathbb{Q}^N$.  
In view of the experience gathered beginning from (\ref{domcvgenrmlfrm}), we may surmise that: the constant terms in the defining functions of the above collection of supporting hyperplanes to $G=\lambda(D)$,
would conceivably -- a rigorous presentation is forthcoming -- yield the coefficients of a power series convergent on $D$. As these constant terms ought to be the values of the support function $h$ for $G$ on a countable dense subset (consisting of the limits of $R_k^n$) of $PS_N$, we may move higher in the ladder of precision. Keeping choices simple, the upshot is that we are led to consider the coefficients as in
the aforementioned prescription: take the coefficient of the monomial $z^{J^{nk}}$ to be $c_{kn} = e^{-l_{kn}h(\alpha^n)}$ with $\alpha^n$ being as in foregoing para, namely the limit of $R_k^n$ as $k \to \infty$. 
The resulting power series ought to work by the following geometric reasoning: as $h(\alpha^n)$ is the distance in the $l^\infty$-metric from the origin to the supporting hyperplane for $G$ with gradient $\alpha^n$ (this was recorded elaborately much earlier as well), it ought to follow that the half-spaces defined by affine functions with gradients $R_k^n$ and with constant terms $c_{kn}$, being close to the supporting half-spaces, must yield the domain
$G$ upon passing to a (suitable) limit; that this indeed does follow is what is demonstrated next.\\

\noindent To work out the aforementioned strategy rigorously, pick a countable dense subset out of the set of all supporting hyperplanes for $G$. Indeed, this may be done by considering hyperplanes defined by affine functions of the form
\[
A_n(s):= \langle \alpha^n, s \rangle - h(\alpha^n) 
\]
where $h$ is the support function of the convex set $G$ and $\{\alpha^n\}$ is  any countable dense subset of $PS_h$. Let us mention in passing that it may well happen that $PS_h$ is just a singleton; indeed, it will be instructive to keep the following example in mind: any complete multicircular domain in $\mathbb{C}^2$ the boundary of whose logarithmic image is a line. Next, the convexity of $G$ and hence of the support function $h$ (and subsequently the continuity of its restriction to $PS_h$ following from the sub-linearity of the support function), forces $G$ to equal the countable intersection of the half spaces $\{ s\in \mathbb{R}^N \; : \; A_n(s)<0\}$. Next, for each $\alpha^n$, choose a sequence $R_j^n$ from $PS\mathbb{Q}^N$ which, as $j \to \infty$, converges to $(\alpha_1^n, \ldots, \alpha^n_N)$ . Then, consider the power series
\begin{equation} \label{splpowseries1}
f(z)= \sum\limits_{j,n \in \mathbb{N}} c_{jn} z_1^{k_{jn} R^n_{j1}} z_2^{k_{jn} R^n_{j2}} \ldots z_N^{k_{jn} R^n_{jN}}
\end{equation}
where $c_{jn}= e^{-k_{jn}h(\alpha^n)}$ with $k_{jn}$ being the least common multiple of the (+ve) denominators occurring in the reduced representation of the rational numbers $\{R_{j1}^n, \ldots, R_{jN}^n\}$. Now, the logarithmic image of the domain of convergence of the power series $f$, which we will denote by $G_f$, can be written using (\ref{domcvgenrmlfrm}) as:
\begin{align*}
&\Big\{ s \in \mathbb{R}^N\; : \; \limsup_{j,n \in \mathbb{N}} 
\Big(\frac{k_{jn} \langle R_j^n,s \rangle + \log \vert e^{-k_{jn} h(\alpha^n)} \vert}{k_{jn}} \Big)<0 \Big \} \\
&=\liminf_{j,n \in \mathbb{N}} \{s \in \mathbb{R}^N \; : \; \langle R_j^n, s\rangle - h(\alpha^n)<0\}
\end{align*}
Thus $G_f$ is the limit infimum of half spaces $H_j^n$ defined by
$B_j^n(s) = \langle R_j^n,s \rangle - h(\alpha^n)$. We wish to compare this representation of $G_f$ with the representation of $G$ as the intersection of half-spaces given by
\begin{equation}\label{Grep}
G= \bigcap\limits_{n \in \mathbb{N}}  \{s \in \mathbb{R}^N  \; :\; A_n(s) <0 \}
\end{equation}
Indeed, to establish the claim that the domain of convergence of $f$ is precisely $G$ or in other words, to show the equality of domains: $G_f=G$, we proceed as follows. Pick any $s^0\in G$.
So $s^0$ belongs to every of the half-spaces appearing on the right of (\ref{Grep}); so $\langle \alpha^n, s^0\rangle-h(\alpha^n)$ is negative. We need to look at
\[
B_j^n(s^0) = \langle R_j^n, s^0 \rangle - h(\alpha^n) = \langle R_j^n - \alpha^n, s^0\rangle + \big(\langle \alpha^n, s^0 \rangle-h(\alpha^n)\big)
\]
Depending on $s^0$ and $n$, choose $j(n, s^0) \in \mathbb{N}$ large enough for $R_j^n$ to be so close to $\alpha^n$ that the second term at the right-most, is bigger in magnitude than its preceding term; more precisely, the `close'-ness and the choice of $j(n, s^0)$ may be made by the following estimation:
\[
\vert \langle R_j^n - \alpha^n, s^0 \rangle \vert \leq \vert R_j^n - \alpha^n \vert \vert s^0 \vert 
< \vert \langle \alpha^n, s^0 \rangle - h(\alpha^n) \vert
\]
which holds for all $j > j(n, s^0)$. This ensures that $\langle R_j^n, s^0 \rangle - h(\alpha^n)$ is negative whenever $j > j(n, s^0)$. However, we cannot immediately claim that $s^0$ lies in all but finitely many of the half-spaces $H_j^n$  so as to conclude that $s^0$ belongs to their limit infimum, $G_f$.
This will follow if we can remove the dependence of $j(n, s^0)$ on $n$. Indeed, it suffices to verify that 
$\vert \langle \alpha^n,s^0 \rangle - h(\alpha^n)\vert$ can be bounded below by a positive constant independent of $n$, for we may always choose the rate of convergence of $R_j^n \to \alpha^n$, to be independent of $n$ -- for instance, we may choose $R_j^n$ so that $\vert R_j^n - \alpha^n \vert<1/j$.  To achieve the desired lower bound, notice first that 
$\langle \alpha^n,s^0 \rangle - h(\alpha^n)$ has a geometric meaning: it is the distance from $s^0$ to the supporting hyperplane $H_\alpha$ for $G$ of gradient $\alpha$, upto a factor of the length of $\alpha$. To be precise and to proceed further, let $s^1$ denote the point where the perpendicular from $s^0$ on the supporting hyperplane  $H_\alpha$ cuts the boundary $\partial G$ -- both the existence and uniqueness of such a point $s^1$ follows from the convexity of $G$; indeed to indicate the key point in the reasoning here,  $H_\alpha$ is contained in the complement of $G$ while both $s^0$ and $\partial G$ are contained in the same one of the
(closed) half spaces determined by $H_\alpha$. 
\begin{figure}
  \includegraphics[scale=0.25]{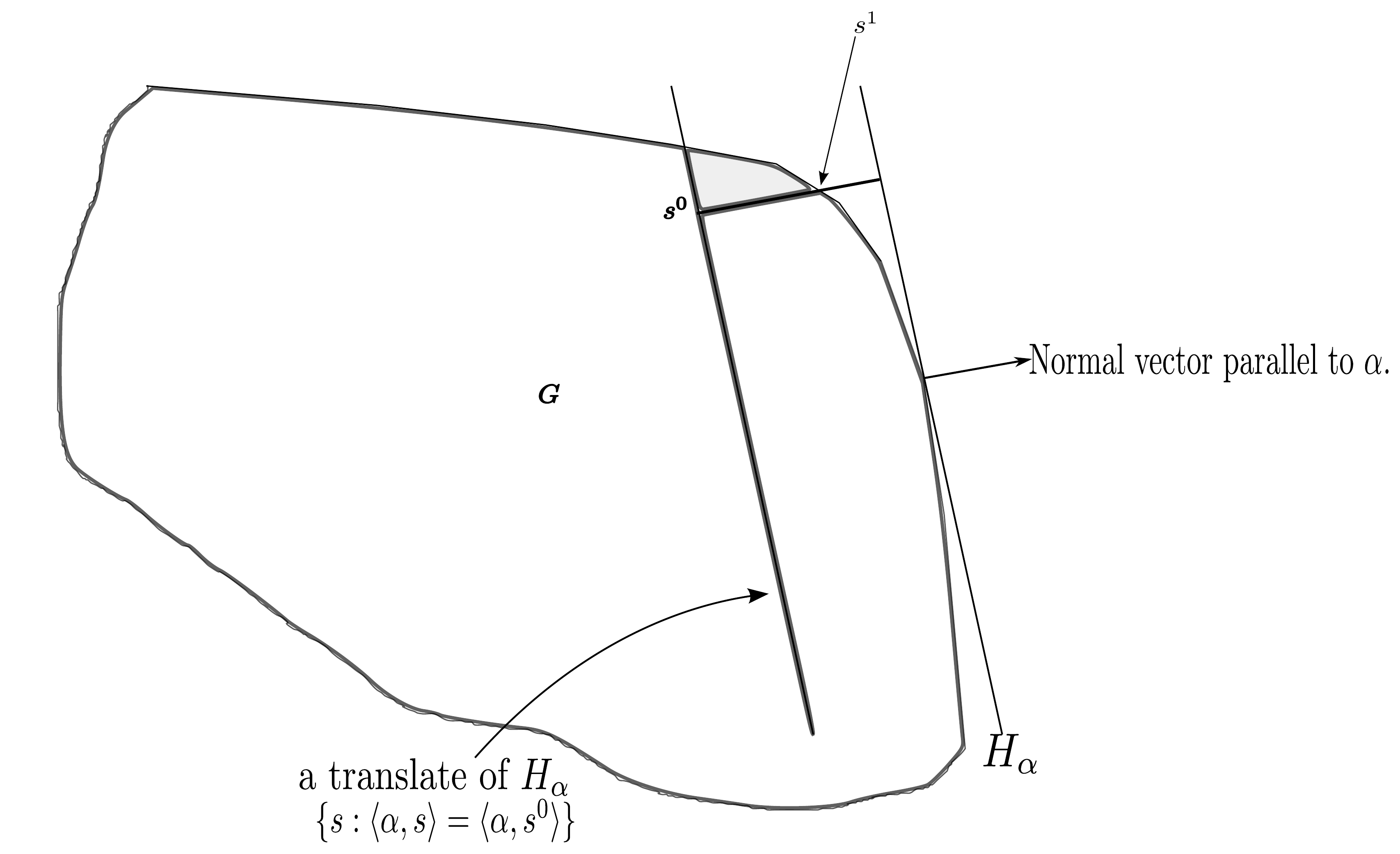}
  \caption{}
  \label{fig: hyperdist}
  \end{figure}
\noindent An illustrative figure such as figure \ref{fig: hyperdist}, convinces us that the distance between $H_\alpha$ and the hyperplane of gradient 
$\alpha$ passing through $s^0$, satisfies the following lower bound: 
\[
\vert \langle \alpha^n,s^0 \rangle - h(\alpha^n) \vert \geq \frac{1}{\vert \alpha^n \vert_{l^2}} {\rm dist}(s^0, s^1).
\]
Recalling that $\alpha^n \in PS_N$ and that $\vert \cdot \vert_{l^2} \leq \vert \cdot \vert_{l^1}$, renders the desired independence of $n$ in the lower bound:
\[
\langle \alpha^n,s^0 \rangle - h(\alpha^n) \geq {\rm dist}(s^0, \partial G).
\]
As noted before this enables us to drop the dependence of $j(n,s^0)$ in the above, which we shall now write as $j(s^0)$. This ensures that for all $n$, $\langle R_j^n, s^0 \rangle - h(\alpha^n)$ is negative except possibly when $j \leq j(s^0)$. Thus, $s^0$ lies in all but finitely many of the halfspaces 
$H_j^n$ whose limit infimum is $G_f$. This is exactly the requirement for $s^0$ to belong to this limit infimum. Thus, $G_f \subset G$.\\

\noindent To obtain the reverse inclusion start again with a point $s^0$, this time in $G_f$. Then, for all but finitely many values of the indices $(j,n)$, we must have
\[
\langle R_j^n, s^0 \rangle -h(\alpha^n)  <0. 
\]
Let \[
\mathcal{R} = \mathcal{R}_G= \{ R_j^n \; :\; \langle R_j^n, s^0 \rangle -h(\alpha^n) \textrm{ is negative}\},
\]
 which differs from the set of all $R_j^n$'s, only by a finite set. Then $\{ \alpha^n \; : \;  n\in \mathbb{N}\}$ is contained in the closure 
 of $\mathcal{R}$. For each $n$, the continuous function
 \[
 \langle \cdot, s^0 \rangle - h(\alpha^n)
\]
 is (finite and) negative on $\mathcal{R}$ and therefore non-positive on $\overline{\mathcal{R}}$. Therefore for every $n \in \mathbb{N}$,
 $\langle \alpha^n, s^0 \rangle - h(\alpha^n)$ must be non-positive. This means that $s^0 \in \overline{G}$ and
  subsequently that $G_f \subset \overline{G}$. As $G_f$ contains $G$ and is open in $\mathbb{C}^n$, it follows that $G_f=G$. \\
  
\noindent While we have established that every logarithmically convex complete Reinhardt domain is the domain of convergence of some power series,
the foregoing considerations do not in anyway mean that to obtain such a domain, all one has to do is to merely take the pull back 
of any arbitrary convex domain in $\mathbb{R}^N$ via the logarithmic map and thereafter by the absolute mapping.
What are the characterizing properties to be possessed by a convex domain $G \subset \mathbb{R}^N$
in order for it to qualify to be the logarithmic image
of the domain of convergence of some power series? 
We seem to be confronted with finding out a way to decide by `looking' at a given convex domain,
if it is indeed the logarithmic image of the domain of convergence of some power series.  We digress a bit for the sake of 
refining our geometric understanding of domains of convergence. 
Logarithmic convexity may not be a property as intuitive as standard geometric convexity; nevertheless, let us not be amiss to note certain 
easy consequential visible properties common to all domains of convergence of power series; for instance: all of them are topologically trivial i.e., are contractible domains. While contractibility alone need not necessarily imply topological equivalence with the ball in general, 
their linear contractibility i.e., starlikeness does. We refer the reader to the last section \ref{last} for a proof.
Consideration of such domains combined with their boundaries will be important as well; in this connection, 
we first remark that the logarithmic images of each such domain has the property that it's
closure is the epigraph of a convex function on the hyperplane  
$\partial \lambda(H) = \{s \in \mathbb{R}^N \; : \; s_1 + \ldots + s_N = 0\}$ where 
$H$ is the domain in $\mathbb{C}^N$ given by $\{ z \; : \; \vert z_1 z_2 \ldots z_N \vert <1\}$. To see this,
recall that a logarithmically convex complete Reinhardt domain $D$ being a star-like domain has associated to it a pair of functions namely, the radial function $R$ and its reciprocal gauge function. It follows from the observations made at around (\ref{-logR}), that $-\log R \circ {\rm Exp}(s)$ 
provides a convex defining function for the domain $G$.   
Further contemplation convinces us that 
$\overline{G}$ must be an unbounded convex body whose boundary is homeomorphic to $\mathbb{R}^{N-1}$. Indeed, $\overline{G}$
can be realized as the epigraph of a convex function on the hyperplane with gradient vector $(1,1,\ldots,1)$ through the origin 
and with $\partial G$ being the graph of such a function. 
This function
is determined by the one-to-one correspondence $\partial \lambda(H)  \to \partial G$
set up by the 
composition of the following maps: first apply
${\rm Exp}$, thereafter $r \to r/\vert r \vert_{l^1}$ followed by
$R(z)z = R(\vert z \vert) \vert z \vert = R(r)r$ which lies in $\tau(\partial D)$ and then finally the mapping $\lambda$. 
This explicitly takes the form
\[
s \to \big(s_1, \ldots,s_N\big) -\log R \big( {\rm Exp}(s)\big) (-1, \ldots,-1)
\]
which means that $\partial G$ is the graph of the function $-\log R \circ {\rm Exp}$ on the hyperplane $\partial \lambda(H)$. 
Seen differently, the 
absolute image of any such domain can be realized as the graph of an extended-real valued function over the probability simplex through the `radial 
function' available for any star-like domain, as indicated in figure \ref{fig:topequiv}. \\

\begin{figure}[h]
  \includegraphics[scale=0.25]{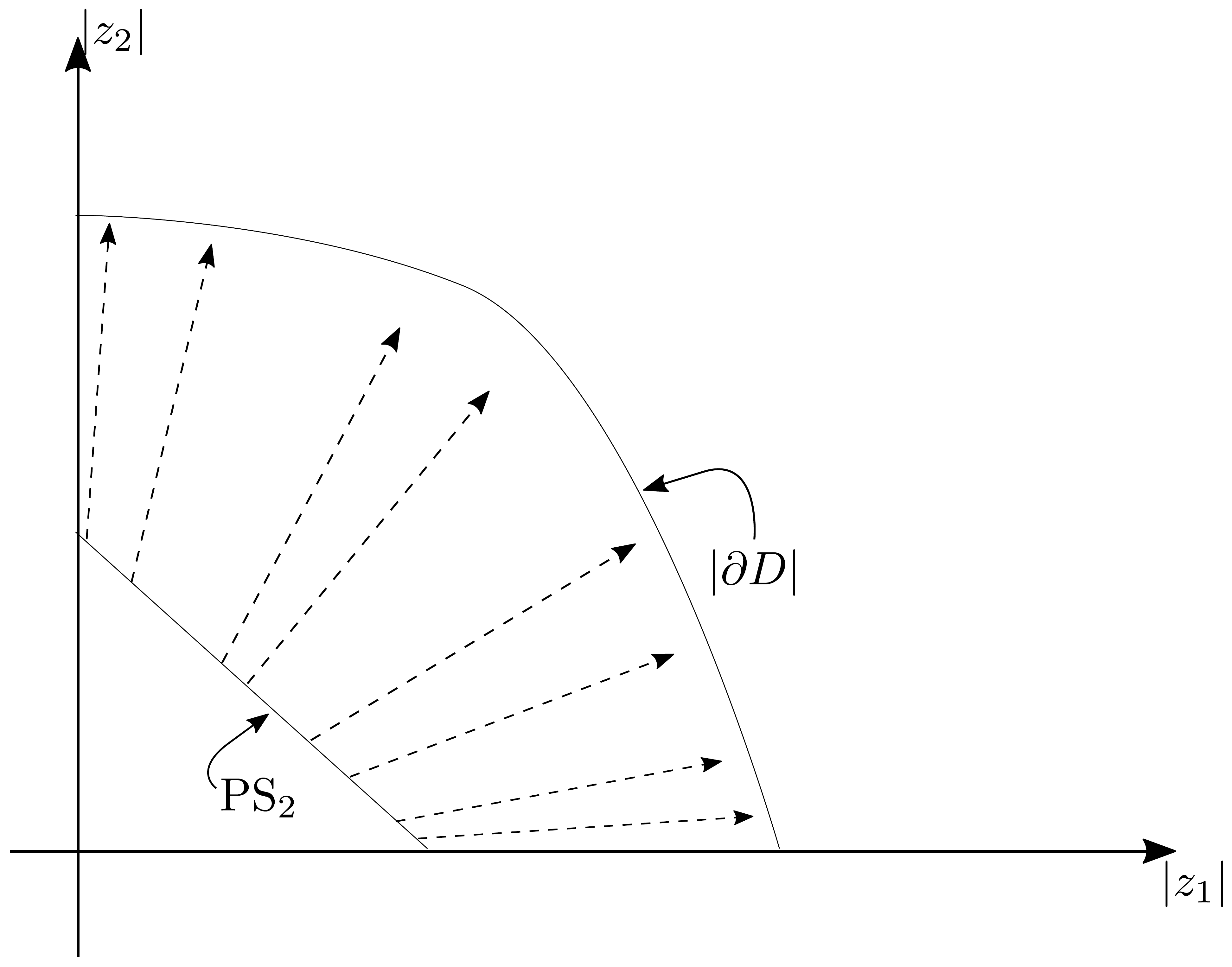}
  \caption{Illustrating the correspondence between the probability simplex and the absolute image of the boundary of a complete Reinhardt domain, when the domain is bounded.}
  \label{fig:topequiv}
\end{figure}

\noindent This common feature leads to domains of convergence of
power series, becoming mutually topologically equivalent provided only that we add points 
at infinity to those among them which are unbounded. To be more precise, it turns out that while starlikeness of a bounded domain alone need not necessarily ensure its clousure to be topologically equivalent to the closed ball, the fact that a domain $D$ of convergence of some power series, say with $D$ bounded for simplicity,
is {\it complete Reinhardt} ensures that its closure $\overline{D}$ is homeomorphic to the closed ball $\overline{\mathbb{B}}^N$. 
However, this does not mean that domains of convergence of power series are holomorphically equivalent to the ball and thereby to each other. Indeed, two of the 
simplest 
logarithmically convex complete Reinhardt domains (which is to say, domains of convergence of power series) namely, the polydisc $\mathbb{U}^N$
whose logarithmic image has its support function finite-valued on all of $PS_N$ and the unbounded domain $H=\{ z \in \mathbb{C}^N \; :\; \vert z_1 \ldots z_N \vert <1\}$  obtained as the inverse image of a half-space under the logarithmic map $\lambda$ with the support function of $H$ finite only at a single point of $PS_N$, are biholomorphically inequivalent.
One way to see this non-equivalence of $H$ with $\mathbb{B}^N$ or $\mathbb{U}^N$ is via a theorem due to H. Cartan about biholomorphic mappings between circular domains, in conjunction with the fact that the automorphism group of $\mathbb{B}^N$ or $\mathbb{U}^N$ act transitively on their respective domains. This failure of the Riemann mapping theorem of complex analysis in dimension one, persists even if we restrict ourselves to bounded domains of convergence of power series in any higher dimension. Indeed, two of the simplest topologically trivial bounded domains $\mathbb{B}^N$ and $\mathbb{U}^N$ are not
biholomorphically equivalent. Cartan's methodology towards establishing these inequivalences, is an excellent example of how the local power series representation of holomorphic functions suffices to provide an elementary and (yet!) neat proof of, the just alluded to failure, of the Riemann mapping theorem in every dimension $N>1$. We mention in passing, as a matter of (a non-trivial!) fact that any pair of such domains (of convergence of some power series) will generically fail to be biholomorphically equivalent.\\

\noindent Let us get back from the detour about gaining an understanding of the topology of domains of convergence, 
to our question: what are the characterizing features of the logarithmic images of domains of convergence of power series?
A little contemplation 
will help convince that what is 
required here is, to determine a condition to be imposed on the given convex domain $G \subset \mathbb{R}^N$
which ensures the completeness of $\lambda^{-1}(G)$, without involving its pull back into the absolute space. An answer to this requirement 
here would be that the 
characteristic/recession cone of $G$ contain the standard backward cone $(-\mathbb{R}_+)^N$. This means that the domain $G$ is unbounded in $N$-many 
independent directions; to be precise, in every direction of $(-\mathbb{R}_+)^N$. 
While what constitutes a satisfactory answer to our question at hand is subjective,  saying that the characteristic cone contain $(-\mathbb{R}_+)^N$
cannot be considered satisfactory; for, one is often `given' a domain by its defining function and it remains therefore to 
figure out ways to find out the characteristic cone from the defining function. Instead of taking up this task in all its generality,
we shall directly lay down the condition in our setting. Namely, given a convex domain $G \subset \mathbb{R}^N$ with defining function $\psi$,
write down a condition to decide if the characteristic cone of $G= \{ s \; : \; \psi(s)<0 \}$ contains $(-\mathbb{R}_+)^N$. As this is
easily seen to be equivalent to requiring that the gradient vector of $\psi$ at all boundary points of $G$
i.e., the outer normal vector field along $\partial G$ points into the standard cone
$\mathbb{R}_+^N$, the sought-after condition on $\psi$ is that it satisfy
\begin{equation*} 
\langle \frac{\nabla \psi}{\vert \psi \vert} , \frac{1}{\sqrt{N}}(1,\ldots,1) \rangle < \frac{1}{\sqrt{2}}
\end{equation*}
which can be rewritten as:
\begin{equation}\label{deffncondn}
2\big(\partial \psi/\partial z_1 + \ldots + \partial \psi/\partial z_N \big)^2 <
 N \big(\vert \partial \psi/\partial z_1 \vert^2 + \ldots + \vert \partial \psi/\partial z_N  \vert^2 \big),
\end{equation}
where all derivatives are to be evaluated at points $s$ in $\mathbb{R}^N$ where $\psi(s)=0$. This is the analytic condition for a convex function $\psi(s)$  
to satisfy, for the convex domain $G = \{ s \; :\; \psi(s)<0\}$ defined by it, to be the logarithmic image of the domain of convergence 
of some power series.\\


\noindent Now, while what we have shown in the foregoing paras, means for instance, that there is a power series convergent precisely on $\mathbb{B}^N$, we have not shown 
that every holomorphic function on $\mathbb{B}^N$ can be represented by a single convergent power series, as in dimension one. In fact, we have thus 
far, not really dealt with `holomorphicity'.

\begin{defn}
Let $D\subset \mathbb{C}^N$ be a domain. A function $f: D \to \mathbb{C}$ is said to be holomorphic if it admits a local representation by convergent power series i.e., every  point $p \in D$ has corresponding to it a countable set of complex numbers $\{c_J(p): J \in \mathbb{N}_0^N\}$ and a neighbourhood $U_p$
such that the power series about $p$, $\sum c_J(p) (z-p)^J$ converges for all $z \in U_p$ to $f(z)$.  
\end{defn}

\noindent Thus a holomorphic function $f$ on a domain $D$ may be thought of as being obtained by gluing together an appropriate collective 
of `function elements' with each such element being defined by power series convergent on some patch (=sub-domain) inside the domain; the appropriateness
here being the requirement of the collective to satisfy basic compatibility conditions: any two members out of this collective need to agree on the intersection of their patches. We shall not digress into complex analysis of several variables here; in particular not even pause to discuss the uniqueness of the numbers $c_J(p)$ in the possibility of multiple local representation by power series in the definition above. We refer the reader to standard references (such as \cite{R} or \cite{S}) wherein familiar basic 
properties such as the (local) Cauchy integral formula, maximum modulus principle, open mapping theorem, identity principle, theorems of Weierstrass and Montel etc., are established for holomorphic functions of several variables; alternative definitions for holomorphic functions are provided and the equivalences established therein as well. We shall only remark that analogous to the one variable case, the numbers $c_J(p)$ are given by: $c_J(p) = D^Jf(p)/J!$. This means that  
local information about $f$ {\it near} any point $p\in D$, is determined by the `germ' of infinitesimal data of $f$ and 
all its derivatives {\it at} the point $p$. Dual to this outward flow of information about $f$ from $p$ is the more interesting
inner sweep: local information about $f$ in a neighbourhood $U$ of the point 
$p\in D \subset \mathbb{C}^N$ can be obtained by suitably integrating the data about the  values of $f$ alone -- no derivatives required -- on the thin subset of the boundary of the polydisc $P$ describing the neighbourhood $U$, given by its distinguished bit $\partial_0 P$ whose real dimension is $N$ (half that of $D$). Indeed, what is being alluded to here is the local Cauchy integral formula valid for polydiscs from which Cauchy estimates follow as
was shown in lemma \ref{Cauchyest} much as in the one variable case. In particular, control on the values of a holomorphic function $f$ and all its derivatives at a
point is attained from the knowledge of its values on the distinguished boundary of any polydisc centered at that point and contained in the domain of $f$.
This facilitates establishing the representation of a function holomorphic on a polydisc by a single power series. Concerning the representation of holomorphic functions by a single power series on discs in dimension $1$, we must remark here that: it should not be concluded from the foregoing considerations it is only on logarithmically convex complete Reinhardt domains that every holomorphic function has a representation by a single power series. Infact, such a representation is valid on any complete Reinhardt domain -- logarithmic convexity is inessential here. This follows from the foregoing observation on the representation of holomorphic functions on a polydisc by a single power series convergent therein, together with the fact that complete Reinhardt domains are nothing but a union of concentric polydiscs. Finally we remark in passing that infact, we may expand any holomorphic function on any complete circular domain, into a series of homogeneous polynomials compactly convergent on such a domain. All this and much more can be found in the excellent text \cite{JarPflu}. \\

\noindent Among the first fundamental and strikingly new phenomenon in complex dimensions $N$ any greater than one, is the Hartogs phenomenon: every holomorphic function on a punctured polydisc extends across the puncture, so that in particular, there are no isolated singularities for holomorphic functions on domains in dimensions $N>1$. For convenience in sketching a quick proof, let us demonstrate this phenomenon on $U^p :=U \setminus \{p\}$ where 
$U$ is the polydisc centered at the origin in $\mathbb{C}^2$ of polyradius $(2,2)$ with the puncture $p=(1,1)$. Given any holomorphic function
$f$ on $U^p$, we apply the fact mentioned in the foregoing para to the restriction of $f$ to the complete Reinhardt domain $L$, where $L$ is the subdomain 
of $U^p$ whose absolute diagram was sketched in figure \ref{fig:Logconvexity}; the aforementioned fact ensures a representation of $f$ by a single power series compactly convergent (at least) on $L$. But then as $L$ fails to be logarithmically convex, we conclude that this power series must converge on some neighbourhood of $p$ as well. As the function defined by this power series already agrees with $f$ on the open set $L$, the identity principle guarantees that this 
is indeed a holomorphic extension of $f$ across $p$. This finishes the proof that all functions holomorphic on the domain $U^p$ extend `simultaneously' 
across the boundary point $p$. \\
 
\noindent It is then natural to single out domains maximal with respect to this phenomenon of simultaneous extension of holomorphic functions i.e., domains $D$ such that for each boundary point $p \in \partial D$, there is a function 
$f_p$ holomorphic on $D$ resisting holomorphic continuation to any neighbourhood of $p$.
A domain possessing this property is called a domain of holomorphy. It turns out that this property is equivalent to the stronger requirement that
there exist at least one holomorphic function which does not extend holomorphically across the boundary near any point in $\partial D$. Infact this is 
only one among multiple equivalent definitions/characterizations of domains of holomorphy. A celebrated problem going by the name of the Levi problem and taking several decades for its complete resolution, was to obtain a geometric characterization of domains of holomorphy. This is best left for another essay; suffice it to say here that the answer lies in a subtle convexity property  
and we refer the reader again to \cite{R1}, \cite{R2} and other texts of the subject. Our next goal here will be to show that domains of convergence of power series are indeed domains of holomorphy. \\

\noindent The question to be dealt with now is: given a domain $D$ which is the domain of convergence of some power series (equivalently, a logarithmically convex complete Reinhardt domain $D$) in $\mathbb{C}^N$ and an arbitrary point 
$p$ of its boundary $\partial D$, is it possible to construct (another) power series $f_p(z)$ which converges on $D$ and whose limit supremum  as $z \to p$ is $\infty$? Note that this question does not get trivially settled with the knowledge of the existence of a power series converging precisely on $D$, owing to the possibility of the existence some (tiny) piece of $\partial D$ across which all such power series can  somehow be continued holomorphically.
As already seen at (\ref{splpowseries1}), while constructing power series with certain desired properties, it is best to use the freedom in expressing them as a sum of monomials in any order that we wish -- in a manner that is telling about the desired properties. With this flexibility, let us demonstrate that domains of convergence of power series are 
(what are known as `weak'-) domains of holomorphy by constructing the function $f_p$ in question. We cannot help but narrate here the clear but concise 
treatment in Ohsawa's little text \cite{Oh}. Suppose first that $D$ is bounded and observe that given any point $p$ in the exterior of $D$ (i.e., $p \in \mathbb{C}^N \setminus \overline{D}$), there exists a monomial $m_p(z)$ such that
\begin{equation} \label{monomial} 
\sup\limits_{z \in D} \vert m_p(z) \vert < m_p(p) =1. 
\end{equation}
Indeed, this follows essentially by passing to the logarithmic image $G=\lambda(D)$, applying to it a standard separation theorem to the convex domain $G$ and then exponentiating back -- the only possible hitch in this process arising when some of the coefficients of the gradient vector of the hyperplane separating $\lambda(p)$ and $G$ are irrational, can be overcome by a slight perturbation of the hyperplane preserving the separating property. The possibility of such a suitable slight perturbation is facilitated by the assumption of the boundedness of $D$ as illustrated in the figure \ref{fig:perturbing}. 

\begin{figure}[h]
  \includegraphics[scale=0.15]{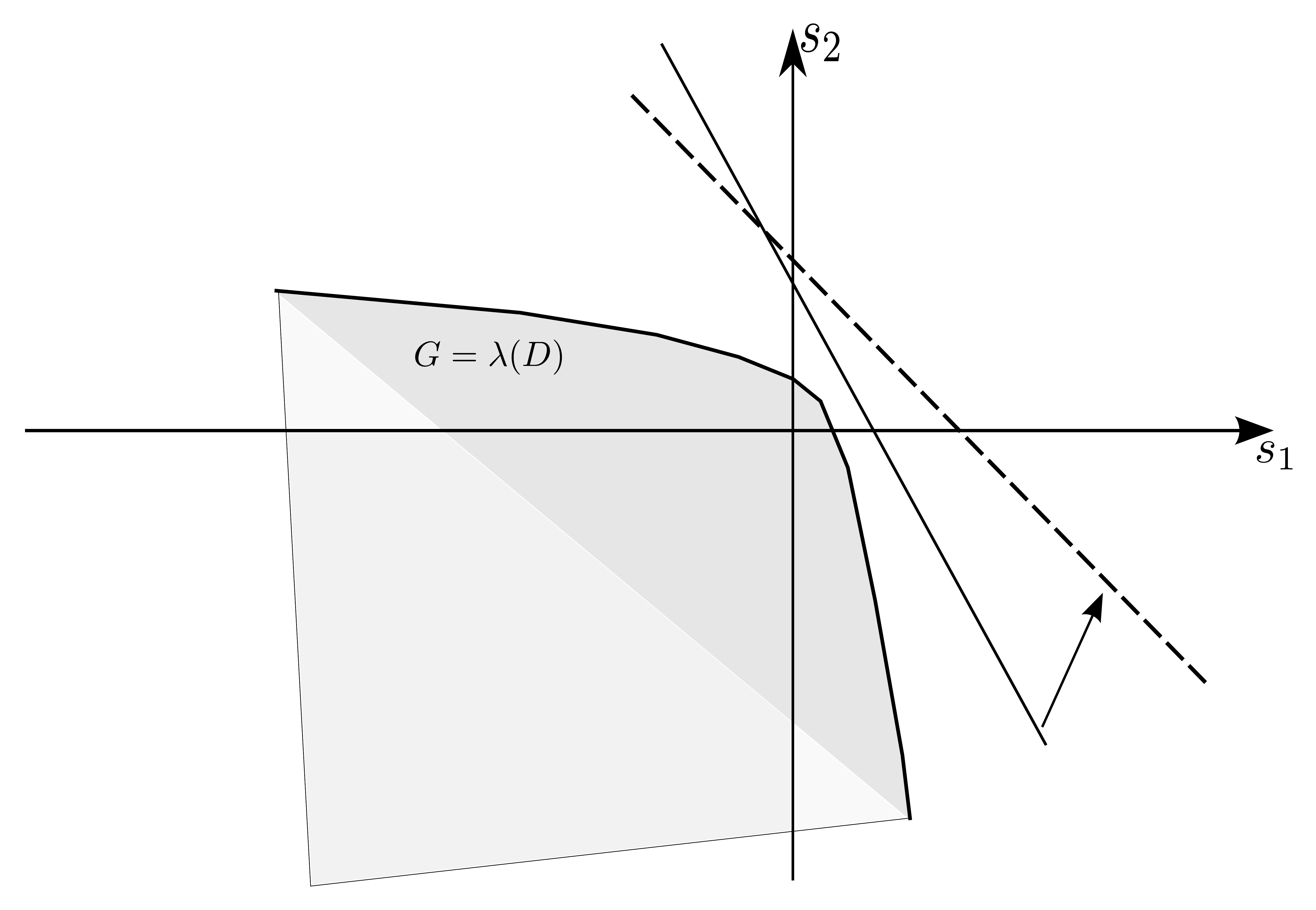}
  \caption{Perturbing to get a rational gradient.}
  \label{fig:perturbing}
\end{figure}

\noindent Among other things, what (\ref{monomial}) means is that we may arrange for the supremum on $D$ appearing therein to be arbitrarily small, by taking powers of the monomial $m_p$,  while maintaining the value at $p$ to be at unity; in symbols, $m_p(z)^{n_k}$ for a suitable $n_k \in \mathbb{N}$, will satisfy
\[
 \sup\limits_{z \in D} \vert \big(m_p(z)\big)^{n_k} \vert < 1/2^k.
\]
The sum of such monomials gives a power series uniformly convergent on $D$.
Further now, we need to modify this series to make it take arbitrarily large values along some sequence 
approaching $p$. Thus on the one hand, we need the supremums on compact subdomains of the monomials constituting our power series to decrease exponentially 
and on the other hand we need its values along some sequence approaching the boundary to blow up. In order to have these requirements met, it is 
natural to exhaust the given domain $D$ by a sequence of relatively compact subdomains expanding out to the boundary and then apply (\ref{monomial}) to each member of this sequence. Before proceeding to work this out rigorously, note that we may further multiply the monomial $m_p$ as above,  by a constant $C$ independent of $k$ to get a monomial, denoted again by $m_p$, which assumes the value $C$ at $p$ and satisfies an exponential decay rate in $k$ on the given domain $D$:
\begin{equation}\label{refined}
 \sup\limits_{z \in D} \vert \big(m_p(z)\big)^{n_k} \vert < C/2^k.
\end{equation}
Now, we may drop the assumption that $D$ is bounded, for we intend to apply (\ref{monomial}) or rather its refined version (\ref{refined}) for our purpose 
only to the bounded {\it subdomains} exhausting $D$ as mentioned above. To work this out, denote by $\mathbb{B}_j$ the ball of radius $j$ about the origin which by the way is recalled to be an interior point of $D$. Figuring out that $\lambda^{-1}(G_j)$ must be connected, where
\[
G_j = \{ s \in \mathbb{R}^N \; : \; {\rm dist}(s, \partial G) > 1/j \},
\] 
we set $D_j$ to denote the bounded sub-domain of $D$ obtained by intersecting the domain $\lambda^{-1}(G_j)$ by $B_j$ -- this intersection has got to be non-empty for all $j$ large and nothing is lost by assuming that this happens right from $j=1$.  Recall that as $D$ is a {\it complete} Reinhardt domain,  the infinite box-neighbourhood of $(-\infty, \ldots, -\infty)$ (at the `left-bottom') arising as 
the logarithmic image of the polydisc spanned by any point 
is contained in $G$ and consequently in all the $G_j$'s as well owing to the concavity of the function ${\rm dist}(\cdot, \partial G)$ on $G$; this ensures that all the $D_j$'s are complete Reinhardt domains as well. 
If $z,w$ are 
a pair of points in $G$ whose distance from $\partial G$ are at least $\delta$, then concavity of the function ${\rm dist}(\cdot, \partial G)$ on $G$,
ensures that the minimum distance of every point of the line segment joining $z,w$ in $G$ lies at a distance at least $\delta$ from $\partial G$. This 
fact ensures that all the domains $G_j$'s are convex and thereby the logarithmic convexity of the $D_j$'s. Thus,
the $D_j$'s form an (increasing) exhaustion of $D$ by {\it bounded} logarithmically convex complete Reinhardt domains. Now, to construct an $f_p$ with 
$\limsup\limits_{z \to p} \vert f_p(z) \vert = \infty$, what could be more simple than to arrange for a function whose values at some sequence $p_j$ of
points in $D$ approaching $p$, is at least as big as $j$? In trying to arrange for such a function $f_p$, we must not loose sight of the requirement that $f_p$ is to be given by a power series which {\it converges on all} of $D$. Recall the availability of a characterizing test to determine whether or not a point 
belongs to the domain of convergence of any given power series, namely proposition \ref{charactest}. 
Put in words, according to this proposition, a point 
 is within the
domain of convergence of a power series if the sequence of complex numbers obtained by evaluating the monomials constituting the power series (in the standard partial ordering by degree) at that point, decays to zero at least at an exponential rate; stated differently, faster than a geometric progression (of ratio $<1$). The last statement holds with the word `point' replaced by `any point from the set of all points whose distance to the boundary of the domain of convergence is bounded below by a positive constant'. We choose the standard geometric progression namely $\{1/2^k\}$ for measuring/controlling the rate in what follows. First, let $p^j$ be sequence in $D$ which converges to $p$; indeed, choose the sequence so that (it escapes out of the $D_j$'s linearly as:) $p^j \in D_{j+1} \setminus D_j$ and converges to $p$. Corresponding to each such $p^j$, by (\ref{monomial}) choose a monomial $m_{p^j}$ whose value at $p^j$ exceeds the supremum of its values on $D_j$. We wish to arrange our series $f_p$ in such a way that the value of the $n$-th term of the series, at $p^n$, exceeds $n$ -- the amount by which it exceeds, is arranged to cancel out the possible negative contributions of the remaining terms, so as
to  ensure ($f_p(n) >n-1$) ultimately that $f_p(p^n) \to\infty$. For instance, we may take the $n$-th term to be $c_n m_{p^n}(z)$ with $c_n >n$, whose value at $p^n$ is $c_n$. The major part of the `negative contributions' to possibly pull down the value of $f_p$ at $p^n$, will conceivably due to the terms preceding the $n$-th term, as the remaining tail of the series $f_p$ (assuming convergence) will be small. Put together with the aforementioned convergence criterion, we are then led to seek for sequences $n_k \in \mathbb{N}$ and real numbers $c_k$ such that
\[
c_k = k +\Big \vert \sum\limits_{j=1}^{k-1}c_j \big(m_{p^j}(p_k) \big)^{n_j} \Big \vert 
\]
together with the requirement
\[
\sup\limits_{z \in D_k} \big \vert c_k ( m_{p^k}(z) )^{n_k} \big \vert < 1/2^k.
\]
It is easy to construct the sequences $c_k$ and $n_k$ inductively, satisfying the above conditions at each stage. Then the series 
\[
\sum\limits_{j=1}^{\infty} c_j \big(m_{p^j}(z)\big)^{n_j}
\]
is compactly convergent (recall $D_j$'s are relatively compact) on $D$ and thus defines a holomorphic function $f_p(z)$ on $D$. As $f_p(p^n) >n-1$, we must have 
$\limsup\limits_{z \to p} f_p(z) = \infty$, with which we have attained our goal of checking out that domains of convergence of power series 
are indeed domains of holomorphy. \\

\begin{rem}
The series just constructed may converge on a domain larger than $D$; so, there is no guarantee that it is also a power series which
converges `precisely' on the given logarithmically convex complete Reinhardt domain $D$.
\end{rem}

\noindent Let us provide for convenience of the reader a concrete power series for the ball:
\begin{ex}
Show that the domain of convergence in $\mathbb{C}^2$ of the power series of two complex variables $z,w$ given by
\[
\sum\limits_{j,k \in \mathbb{N}} \frac{f(j)f(k)}{f(j+k)} z^jw^k
\] 
where $f(t)=\sqrt{t^t}$, is the unit ball $\mathbb{B}^2$. 
\end{ex}

\noindent Taking $f(t)=t^2$ gives a power series whose domain of convergence is precisely 
\[
E_{1/2}=\{(z,w) \in \mathbb{C}^2 \; : \; \sqrt{\vert z \vert} + \sqrt{\vert w \vert} <1 \}.
\]
Thus, this gives an example of a complete Reinhardt logarithmically convex domain which is not convex. Infact, $E_{1/2}$ is a pseudoconvex domain cannot be mapped onto a convex domain by any biholomorphic transformation whatsoever. Indeed, one will need some property invariant under biholomorphic transformations to establish the inequivalence of $E_{1/2}$ with any given convex domain; obviously, it has got to be stronger than merely being a
topological property. It turns out that each domain in $\mathbb{C}^N$ has certain intrinsic complex geometric properties, which remain invariant under biholomorphic mappings and which are captured by what go under the collective
title of `invariant metrics'. For the problem at hand, the technique of invariant metrics reduces the proof for the non-existence of a 
biholomorphism between $E_{1/2}$ and a convex domain to that of a linear mapping. To tell the basic idea, a bit
more precisely, equip the domains with an invariant metric for definiteness, say the Kobayashi infinitesimal metric.
Instead of digressing here into an exposition on invariant metrics or list all
the texts available or even their definition, we shall only mention the introductory article \cite{IKra} 
and get to the role played by them here.
If there were a biholomorphism between $E_{1/2}$ and a convex domain then firstly, note that the linear mapping 
given by the derivative of the biholomorphism renders an equivalence between
the Kobayashi indicatrices at the (tangent spaces at the) corresponding points.
In view of results in \cite{BS}, it is actually possible to assume, after composing with another biholomorphism
if necessary, that the convex domain is bounded.
It is known, due to a deep work \cite{L} of Lempert, that the
Kobayashi indicatrix of any bounded convex domain is convex. On the other hand, the indicatrix
at the origin for $E_{1/2}$, being a copy of $E_{1/2}$ itself, is non-convex. 
As convexity remains invariant under linear mappings, this yields a contradiction. 
The Kobayashi metric is only one of various functorial constructs going by the collective name of `Schwarz -- Pick' systems,
which provide metrics well-adapted for complex analysis; the extent to which this subject has evolved is evidenced by  
the authoritative work \cite{JarPflu2}. 
The example $E_{1/2}$ above already calls attention to the fact that logarithmic convexity
or pseudoconvexity are not naively biholomorphically modified versions of convexity. Infact, the relationship between pseudoconvexity and convexity
is still a mystery -- there remains open till
date, a tantalizing question brought out in the (end of) the article \cite{R1}. We shall round off the discussion here
with a remark about a source for open questions about power series.

\begin{rem} The fact that the subject of power series is fundamental and elementary does not mean that all basic questions about them have more or less been settled. Among many recent works concerning power series, we call attention to the semi-expository article \cite{Bo2} concerning the Bohr phenomenon arising out of functions defined 
by power series on logarithmically convex complete Reinhardt domains; associated to such domains are certain curious numbers called the `Bohr radius'. For
an exposition of this as well as for open problems, the ambitious reader may consult \cite{Bo2}. 
\end{rem}

\noindent Now, we know that given any power series, we may read off the equation defining the boundary of its domain of convergence from its coefficients; it is given precisely by (\ref{deffnforbdy}). Conversely, we have been discussing methods to explicitly write down power series which converge on any given logarithmically convex multicircular domain. Now, when we say, we are `given a domain', what could this mean in practice? The most tangible meaning would be that we are given (the knowledge of all connected components of) the boundary of the domain as the zero set of a defining function. How does one plot points of the boundary, given the defining function $\varrho$, say? Well, the immediate answer would be write down solutions to the equation $\varrho=0$. But then, such an equation is in general is never going to be linear and very likely, difficult to solve.
One way out of this problem, while dealing with convex domains and thereby for our problem of constructing power series,
is to express everything in terms of the support function (as we have already done) and then seek a link between the support function and the defining function, which is the matter that we take up next. \\

\noindent Suppose $G \subset \mathbb{R}^N$ is a convex domain with support function $h$. Then $G$ can be written as the intersection of {\it open} 
half-spaces
\[
G = \bigcap\limits_{\alpha \in \mathbb{R}^N} \{ x \in \mathbb{R}^N   \; : \;  \langle \alpha, x \rangle - h(\alpha) <0 \}
\]
However, we cannot claim from this that $G$ equals $\{ x   \; : \; \sup_{\alpha \in \mathbb{R}^N} \{ \langle \alpha, x \rangle - h(\alpha)\} <0 \}$ nor 
that it equals $\{ x   \; : \; \sup_{\alpha \in \mathbb{R}^N} \{ \langle \alpha, x \rangle - h(\alpha)\} \leq 0 \}$. On the other hand, we may restrict the parameter $\alpha$ to vary over the {\it compact} set $S_N$, the standard simplex, and still write
\[
G = \bigcap\limits_{\alpha \in S_N} \{ x \in \mathbb{R}^N   \; : \;  \langle \alpha, x \rangle - h(\alpha) <0 \}
\]
That is, $G$ equals the set of all those points $x$ which satisfy $\langle \alpha, x \rangle - h(\alpha) <0$ for all $\alpha \in S_N$. So, for each fixed 
$p \in G$, the function $\langle \alpha, p \rangle - h(\alpha) $ is an  upper-semicontinuous concave function which is  strictly negative on $S_N$  and therefore attains its supremum on $S_N$ at some point therein and consequently this supremum must be strictly negative. This proves that
\[
G= \{ p   \; : \; \sup_{\alpha \in S_N} \{ \langle \alpha, p \rangle - h(\alpha)\} <0 \},
\]
a claim that cannot be made if $\alpha$ were allowed to vary over all of $\mathbb{R}^N$ in the above. In other words, this is saying that 
$G$ is precisely the domain defined by the Legendre transform (also called Fenchel -- Legendre transform or convex conjugate) of the restriction of the support function of $G$ to $S_N$. On the other hand, given a defining function $\psi$ for a convex domain $G$ in $\mathbb{R}^N$, it is straightforward to write down the value of the support function for the normal vector at boundary points $p \in \partial G$, as:
\[
h(\triangledown \psi(p)) = \langle p, \triangledown\psi(p) \rangle
\]
which agrees with the Legendre transform of $\psi$ for normal vectors at all points of the boundary. If we normalize the normal vectors at all points of $\partial G$, so as to be unit vectors in the $l^1$-norm, we obtain a convex subset of $S_N$, by virtue of the convexity of $G$. We may then extend $h$ by the general property of positive homogeneity of the support function to obtain its values on a convex cone and subsequently thereafter, take the lower semicontinuous regularization, to completely obtain the support function 
$h:\mathbb{R}^N \to (-\infty, +\infty]$ from a given defining function $\psi$ for $G$.
The Legendre transform, among other notions of duality, is of fundamental importance in the subject of convex analysis which we shall 
only briefly review in the next and last section, and end. \\

\section{Appendix -- Basics of Convex Analysis and Affine Geometry} \label{last}
\noindent The reader is assumed to have some familiarity with convexity. So instead of saying that a convex set is a subset some of $\mathbb{R}^N$ of closed under the geometric operation of formation of straight line segments joining any pair of its points, we are going to say: that a convex set is a subset $C$ of some
$\mathbb{R}^N$, which is closed under the one-parameter family of algebraic operations given by the weighted arithmetic mean 
$(p,q) \to (1-t)p + tq$ 
for $t \in I$ and $p,q \in C$. Our purpose here is to gather together results in convex analysis to serve as a convenient reference for the main text. Proofs therefore, are omitted. They can be found in the systematic treatment \cite{H} or in many good expository texts such as \cite{MT}.
Henceforth $V$ shall denote a real vector space of finite dimension. Given an arbitrary subset $E$ of $V$,
the intersection ${\rm ah}(E)$ of all affine subspaces containing $E$ is an affine subspace called the affine hull of $E$, which has the following analytic expression
\[
{\rm ah}(E)= \{ \sum\limits_{j=1}^{n} \lambda_j x_j \; :\; \sum\limits_{j=1}^{n} \lambda_j=1, \; x_j \in E, \; n=1,2 \ldots \},
\]
If the $\lambda_j$'s in the above are further required to be positive, we obtain what is called the convex hull of $E$, denoted ${\rm ch}(E)$.
Let $C \subset \mathbb{R}^N$ be convex. A point $x$  is said to be in the relative interior of $C$ if $x$ has a neighbourhood $U$ open in $\mathbb{R}^N$ such that $U \cap {\rm ah}(C) \subset C$. Note that the relative interior of a convex set is always a (non-empty) convex set and the closure of the relative interior of $C$ is the closure of $C$. 
A point $x \in C$ is said to be an extreme point of $C$ if it does not lie 
in the relative interior of any line-segment in $C$. The set of all extreme
points of $C$ is denoted ${\rm ext}(C)$.
\medskip \\
\noindent Trivially, every affine subspace of $V$ is convex. An affine subspace of codimension $1$ is termed a hyperplane, which divides $V$ into two connected components; each of these connected components of the complement of a hyperplane is an open {\it half-space}. Each half-space is convex and is denoted generally by $H$ overloaded by some subscript or superscript when it is desirable to specific about its gradient or a point through which it passes. The closure of a half-space -- often denoted by $H$ again -- is convex, as is more generally the closure of any convex set. Another fundamental example of convex set is provided by the class of convex cones: a cone is any set set $A$ which is invariant under homotheties i.e., $x \in A \Rightarrow 
 \alpha x \in A$ for all $\alpha \geq 0$; therefore, convex cones are those cones which are convex sets. One way of generating examples of convex cones,
is to take any set $S \subset \mathbb{R}^N$ and form its 
{\it characteristic/recession} cone, given by
\[
{\rm rec}(S) = \{ y \in \mathbb{R}^N \; :\; x+ \lambda y \in S,
\text{ for all} x \in S \text{ and } \lambda>0 \}
\]
In other words, each vector of ${\rm rec}(S)$ represents a `direction to infinity in $S$.' The {\it lineality space} of the set $S$, denoted
${\rm lin}(S)$, is defined 
to be the largest linear subspace $L$ of $\mathbb{R}^N$ such that
$x + L \subset S$ for any choice of $x \in S$; this can be expressed in 
terms of characteristic cones as:
\[
{\rm lin}(S) = {\rm rec}(S) \cap {\rm rec}(-S).
\]
These notions aid in formulating general structure theorems for  {\it unbounded
convex bodies} i.e, closed sets which are closures of unbounded convex domains. Instances of such bodies of importance for us are logarithmic images of domains of convergence of power series whose boundaries, as noted in the text of the foregoing section, are homeomorphic 
to $\mathbb{R}^{N-1}$. 
An unbounded convex body whose boundary is homeomorphic to $\mathbb{R}^{N-1}$, is by lemma 2.2 of \cite{Ghomi}, expressible
as a sum of its lineality space and its orthogonal projection onto
the ortho-complement of the lineality space, a line-free unbounded
convex body; further, the boundary of the latter summand is also homeomorphic to some Euclidean space. A line-free closed convex set $C$ may be expressed as a sum of its 
characteristic cone and the convex hull of its extreme points ${\rm ext}(C)$.
We shall not spend any further space about such results or notions of
convex bodies, which are strictly speaking not needed here.
However, they may provide helpful background and in this regard 
we find the availability of good expositions of basics of convex bodies in sufficient abundance,
including the systematic encyclopaedic volume \cite{Sch} which however for the most part restricts attention to compact
convex bodies. 
Basics of unbounded convex bodies with proofs and references 
can be found in \cite{Ghomi} or from the much older article \cite{St}. 
Among basic examples 
of bounded convex sets are balls with respect to any norm. Of course all norms on the finite dimensional $V$ are equivalent; but they are far from being affinely equivalent -- note that convexity is preserved by invertible affine maps of $V$ -- in the sense that one cannot be obtained from 
the other by an affine change of variables. Convexity may also be considered on spheres. To introduce this
quickly, consider the unit sphere centered at the origin in $\mathbb{R}^N$ denoted $S^{N-1}$ and a subset $C$ thereof. If the cone 
consisting of rays through the origin and passing through points $p$, as $p$ varies through $C$, happens to be a convex cone in $\mathbb{R}^N$, then 
we say that $C$ is a spherically convex subset of $S^{N-1}$. A spherical-geometric way of checking convexity of a subset $A$ of the sphere, is to check 
for every pair of points $x,y \in A$ with $y \neq \pm x$ that, the set $A$ contains the smaller arc of the great circle on $S^{N-1}$ connecting $x$ and $y$.
We next pass onto the notion of convex functions.
\begin{defn}
Let $X$ be a convex set. A function $f: X \to (-\infty,+\infty]$ is termed convex if 
\[
f(\lambda_1 x_1 + \lambda_2 x_2) \leq \lambda_1 f(x_1) + \lambda_2 f(x_2)
\]
for all pairs of positive numbers $\lambda_1, \lambda_2$ with $\lambda_1 + \lambda_2 =1$ and $x_1, x_2 \in X$. Equivalently, a function is convex iff 
its epigraph 
\[
\{ (x,t) \in V\oplus \mathbb{R} \; :\; x\in X, t \geq f(x) \}
\]
is a convex set.
\end{defn}

\begin{thm}
If $f$ is a convex function on $V$, then 
\[
X = \{ x \in V \; : \; f(x) < \infty \}
\]
is a convex set and $f$ is continuous in the relative interior of $X$ i.e, in the interior of $X$ in ${\rm ah}(X)$.
\end{thm}

\begin{rem}
It is not always possible to redefine $f$ at boundary points of $X$ in ${\rm ah}(X)$, so as to have $f$ become
continuous with values in $(-\infty, +\infty]$.
\end{rem}

\noindent This problem is redressed by taking the lower-semicontinuous regularization.
\begin{prop}
Let $f$ be a convex function on $V$. Define for all $x \in V$:
\[
f_1(x)= \liminf\limits_{y \to x} f(y)
\]
Then $f_1$ is convex and $f_1(x) \leq f(x)$ for all $x$, with equality if $x$ lies in the interior of $X=\{ x \in V : f(x)<\infty\}$ in 
${\rm ah}(X)$ or interior in $V \setminus X$. The function $f_1$ is lower semi-continuous and is termed the lower semi-continuous regularization of $f$.
\end{prop} 
\noindent If $f$ is not given to be defined on all of $V$ but given on a convex set $X$, we first extend by setting its values equal to $+\infty$ at all points  where it is not apriori given i.e., on $V \setminus X$; the above proposition then applies to furnish its lower semicontinuous regularization. The role of lower semi-continuity here is explained as follows. While the epigraph of a function $f$ is convex iff its epigraph is convex,
the epigraph is  closed iff $f$ is lower semi-continuous. This will be important in the subsection on the Legendre transform. 

\begin{defn}
Let $E \subset V$. The indicator function $I_E$ is the function whose value at points of $E$ is set equal to $0$ and equal to $+\infty$ at all points outside $E$. Such a function is convex precisely when $E$ is convex.
\end{defn}

\subsection*{Separation theorems}
\noindent The following four results go by the name of Hahn -- Banach theorems.
\begin{thm}
Let $D$ be a convex domain in $V$. If $x_0 \not \in D$, there is an affine hyperplane $H$ such that $x_0 \in H$ but $H \cap D = \emptyset$. 
Thus there is an affine function $f$ on $V$ with $f(x_0)=0>f(x)$ for all $x \in D$.
\end{thm}

\begin{cor}
Let $X$ be a closed convex subset of $V$. If $x_0 \not \in X$, there is an affine hyperplane containing $x_0$ which does not intersect $X$
i.e., there is an affine function $f$ with $f(x)\leq 0 < f(x_0)$ for all $x \in X$.
\end{cor}
\begin{cor}
If $X$ is a closed convex subset of $V$ and if $y$ is on the boundary of $X$, then one can find a non-constant affine function $f$ such that
$f(x)\leq 0=f(y)$ for all $x \in X$. The affine hyperplane $\{ x \in V \; : \; f(x)=0\}$ is called a supporting hyperplane of $X$.
\end{cor}
\begin{cor}
An open (resp. closed) convex set $K$ in a finite dimensional vector space is the intersection of the open (resp. closed) half-spaces containing it.
\end{cor}

\noindent As a closed convex set is the intersection of its supporting half-spaces, such a set can alternatively be described by specifying the position of its supporting hyperplanes, given their gradient vectors. This is captured by the {\it support function} introduced in definition (\ref{suppfn}). The geometric meaning of the support function is: for a unit vector $u$ with $h(u)$ finite, the number $h(u)$ is the signed distance of the supporting hyperplane to $C$ with normal vector $u$, from the origin; the distance is negative if and only if $u$ points into the open half-space containing the origin. From the definition, it is straight-forward to check that $h_C(\cdot) = \langle z, \cdot \rangle$ is a linear functional iff $C$ is a singleton. More importantly, $h$ is {\it positively homogeneous}:  $h(\lambda u) = \lambda h(u)$ for all $\lambda \geq 0$  and is sub-additive:
\[
h(u +v) \leq h(u) + h(v).
\]
These conditions constitute what is sometimes referred to as sub-linearity, from which it follows in particular that $h$ is a convex function. If $x \in \mathbb{R}^N \setminus C$, a separation theorem yields the existence of a vector $u_0$ with
$\langle x, u_0 \rangle > h(u_0)$. The support function of a convex set $C$ may also be defined as the Legendre transform of its indicator function $I_C$; the Legendre transform being defined in the following subsection.

\subsection*{The Legendre transform}
\begin{defn}
Let $f: \mathbb{R}^N \to \mathbb{R}$ any function. The Legendre transform ($=$ Fenchel -- Legendre transform), also called the convex conjugate, of $f$,  is defined by
\[
f^*(y) = \sup_{x} \{ \langle x, y \rangle - f(x) \}.
\]
\end{defn}
\noindent  We restrict attention to taking convex conjugates only of convex functions as in this case we have the the following key result: The convex conjugate of the convex conjugate of any given convex function is the given function itself. This is only recorded differently, in theorem \ref{CC} below. We have been silent about the domain of $f^*$; we shall allow $+ \infty$ to be in the range of $f^*$. Actually it is convenient here to have  functions defined on all of our vector space and in order to do this, we extend them by setting them equal to $+ \infty$ outside  the convex hull of the set of all points where it's value is specified.  Let $f$ be a convex function such that the set $X$ of all points where it is finite, has non-empty interior which we denote by $X^0$. Then it is possible to argue that $X^0$ must be a convex domain (the basic idea can be found in lemma 11.2.4 of
\cite{Berg}) and $f$ must be continuous herein. Next and further, by taking a liminf of $f$ at points on the boundary of $X^0$ we may redefine $f$ at these points, so that it becomes a lower semi-continuous function on the whole. This will pave the way for using the above definition of the Legendre transform for functions not apriori given to be defined on all of $\mathbb{R}^N$ and more importantly take the domain of $f^*$ to be all of $\mathbb{R}^N$.
\begin{thm} \label{CC}
The Legendre transform is an involution on the space of all lower semi-continuous convex functions on $\mathbb{R}^N$.
\end{thm}
\noindent Thus, 
\begin{equation}\label{frecov}
f(x) = \sup_{m} \{ A_m(x) \}
\end{equation}
where $A_m(x) = \langle m, x \rangle - f^*(m)$.\\

\noindent As a corollary to the foregoing theorem, one may derive another fundamental fact: every `sub-linear' function on a finite dimensional real vector space $V$ arises essentially as the support function of a closed convex set.

\begin{thm}
If $C\subset \mathbb{R}^N$ is a closed convex set, then its support function is lower semicontinuous, convex and positively homogeneous. \\
Conversely, every lower semicontinuous function $h$ on $\mathbb{R}^N$, which is  positively homogeneous and convex (equivalently, positively homogeneous and subadditive) is the supporting function of one and only one closed convex set $C$, given by
\[
C = \big\{ x \in \mathbb{R}^N \; : \; h(v) \geq \langle v, x \rangle \text{ for all } v \in \mathbb{R}^N \big\}.
\]
\end{thm}
\noindent We remark in passing to the next sub-section that, if  $\varrho$ is a defining function for a convex domain $G$, 
the support function of $G$ is given by the Legendre transform of $I_{\mathbb{R}_-} \circ \varrho$, where $I_{\mathbb{R}_-}$ is the indicator function of $\mathbb{R}_-$, the ray of non-positive reals; while this remark may not be useful, the concept of defining function surely is, which we review next.

\subsection*{Defining functions for convex domains.}
\begin{thm}
Let $D \subset\mathbb{R}^N$ be a convex domain. There exists a convex function which is negative on $D$, vanishes precisely on $\partial D$ and is positive on the complement of $\overline{D}$.
\end{thm}

\begin{proof}
Let $p$ be an arbitrary point of $\partial D$. The convexity of $D$ guarantees the existence of a supporting hyperplane for $D$ at $p$ i.e., an affine subspace $L$ of $\mathbb{R}^N$ of codimension $1$ through $p$ with $D$ contained entirely in one, out of the $2$ connected components of $\mathbb{R}^N \setminus L$. Now, if we let $a_p(x)$ denote the affine function which defines $L$, then after a change of sign if necessary we may -- and will! -- assume that $a_p$ is negative throughout $D$. Needless to say, $a_p(p)=0$. Now denote by $\mathcal{F}$ the family of all such affine functions $a_p$ as $p$ varies through $\partial D$. Let
\[
A(x) = \sup_{\mathcal{F}} \{ a_p(x) \}.
\]
Clearly, $A(x)$ is a convex function which is non-negative on $\overline{D}$ which vanishes precisely on $\partial D$. Further, by invoking a suitable separation theorem, we may assure ourselves that $A$ is actually positive on all of $\mathbb{R}^N \setminus \overline{D}$.
\end{proof}

\noindent We shall refer to the function guaranteed by the above theorem as a defining function. With some regularity assumptions about 
the boundary of the domain, it is natural to impose further conditions on the defining function so that it encodes the additional regularity
features. A customary definition for defining functions for smoothly bounded domains -- not necessarily convex -- is as follows:
\begin{defn}
Let $D$ be a domain in $\mathbb{R}^N$. Then $D$ is said to have smooth boundary, if there exists a smooth function $\varrho: \mathbb{R}^N \to \mathbb{R}$ such that  $\varrho$ is positive on the complement of $\overline{\mathbb{D}}$,
\[
D = \{ x \in \mathbb{R}^N \; :\; \varrho(x) <0\},
\] 
$\varrho$ vanishes precisely on $\partial D$ and its gradient vector is non-zero at all points of $\partial D$.  The function $\varrho$ is said to be a
(global) smooth defining function.
\end{defn}

\noindent It is not necessary to have $\varrho$ defined on all of $\mathbb{R}^N$, a tubular neighbourhood of $\partial D$ will suffice; there is also the notion of a local defining function and how one may obtain a global defining function by patching together local defining functions via 
standard partition-of-unity techniques and other results about the relationships between any two defining functions. These matters can be found in standard texts; a reference relevant for the present subsection is 
\cite{HM}. We shall only state the condition of convexity for smoothly bounded domains formulated via the defining function as:
\begin{thm}
Let $D$ be a domain in $\mathbb{R}^N$ with $C^2$-smooth boundary. Let $\varrho$ be a $C^2$-defining function for $D$ near $p\in \partial D$. Then there is an open ball $U$ centered at $p$ such that $U \cap D$ is convex if and only if the Hessian of $\varrho$ satisfies the condition:
\[
\sum\limits_{j,k=1}^{N} \frac{\partial^2 \varrho(p)}{\partial x_j \partial x_k} v_j v_k \geq 0
\]
for all $p \in \partial D$ and $v \in T_p(\partial D)$.
\end{thm}
\noindent Thus at least for domains whose boundaries are $C^2$-smooth, there is a simple analytical local characterization of convexity and their convexity is determined by their boundaries. 

\subsection*{Star-like sets and gauge functionals.}
\begin{defn}
A subset $S$ of $\mathbb{R}^N$ is termed star-like with respect to some point $p \in S$ if the line segment joining any point of $S$ to $p$ is contained in $S$. 
\end{defn}
\noindent The point $p$ as in the definition above, is sometimes referred to as a center of $S$. The set of such centers for $S$ form a convex subset of $S$. The simplest star-like sets are the convex sets, being star-like with respect to each of its points.
We shall henceforth deal only with star-like subsets $S$ which have origin as one of its centers. Infact, we shall restrict attention only to star-like domains $D$ with respect to the origin. For such a domain $D$, first define the {\it radial function}
 $\rho_D: \mathbb{R}^N \setminus \{0\} \to \mathbb{R}^+ \cup \{+\infty\}$ by
\[
\rho_D(x) = \sup \{ \lambda \in \mathbb{R}_+ \; : \; \lambda x \in D \}
\]
which is a strictly positive function owing to the fact that the origin is an {\it interior} point of $D$ and takes the value $+\infty$ precisely when $D$ is unbounded. Further note that $\rho = \rho_D$ is homogeneous of degree $-1$ i.e., $\rho(t x) = t^{-1} \rho(x)$ for all $t \in \mathbb{R}^+$ and the (scaled) point $\rho_D(x)x$ lies 
on the boundary $\partial D$ for all $x\neq 0$. Next, define the {\it gauge functional} (also called the Minkowski functional) of the domain $D$,
$g_D: \mathbb{R}^N \to \mathbb{R}_+$ first for non-zero vectors as the reciprocal of the radial function:
\[
g_D(x) = \frac{1}{\rho_D(x)}
\]
and then setting it equal to $0$ for its value at the origin. Thus, while $\rho_D$ may take on $+\infty$, the gauge function is always finite-valued. We shall merely list the properties of this well-known function and show only the relatively non-trivial property (iii) of this list below.
\begin{thm}
Let $D$ be a star-like domain, with respect to the origin in $\mathbb{R}^N$. Then its gauge function $g=g_D$ has the following properties:
\begin{enumerate}
\item[(i)] $g(tx) = t g(x)$ for all $t \in \mathbb{R}_+$ and $x\in \mathbb{R}^N$,
\item[(ii)] $D = \{ x \in \mathbb{R}^N \; : \; g(x)<1\}$,
\item[(iii)] $g$ is upper semicontinuous,
\item[(iv)] Convexity of $D$ is equivalent to $g$ satisfying the triangle inequality.
\item[(v)] If $D$ is convex, then $g$ is continuous.
\end{enumerate}
\end{thm}

\noindent The gauge functional of $D$ is infact determined uniquely by the properties $(i)$ and $(ii)$ of the listing above. An interesting issue about a star-like domain is notwithstanding the description of the domain by its gauge functional as in $(ii)$, it is in general not possible to claim an equality
in the following trivia: $\{ x \in \mathbb{R}^N \; : \; g(x)=1\} \subset \partial D$. This is indeed true when $D$ is convex and more generally for 
{\it proper} star domains i.e., domains $D$ star-like with respect to the origin for which every ray intersects the boundary $\partial D$ at no more than one point; in such a circumstance, claim can be made about the continuity of $g_D$ as well. Before getting to that, let us first establish the minimal regularity claim that we can make of such functions: 

\begin{lem}
Let $D$ be any star domain, not necessarily bounded, star-like with respect to the origin in $\mathbb{R}^N$. Let 
$\rho= \rho_D: \mathbb{R}^N \to \mathbb{R}_+ \cup \{\infty\}$ denote its radial function. Then $\rho$ is lower semicontinuous.
\end{lem}

\begin{proof}
\noindent We shall only verify the lower semicontinuity of $\rho$ on a small ball $B$ about the origin contained in $D$. This suffices as $\rho$
is homogeneous of degree $-1$. Note that $\rho>1$ throughout $B$ as $B \subset D$. So, let $0 \neq x_0 \in B$ and consider first the case when $\rho(x_0)$ 
is finite. Suppose, to argue by contradiction, that for some $\epsilon >0$, there were a sequence $x_j \to x_0$ with
\[
\rho(x_j)< \rho(x_0) - \epsilon
\]
for all $j$. The idea in a nutshell is this: as $\rho(x)x$ is the last point on the ray $R_x$ through $x$ (originating at the origin ofcourse!) which is adherent to $D \cap R_x$, the above inequality must violate the fact that $\tilde{x}_0=\rho(x_0)x_0$ can be approached by a sequence of (interior) points of $D$, to contradict that $\tilde{x}_0 \in \partial D$. We now proceed towards transforming this
idea into a proof. Let $\hat{x}_0 := \delta \tilde{x}_0$ with $\delta$ a positive number less than $1$, so that $\hat{x}_0 \in D$; a more precise choice for the number 
$\delta$ shall be specified later. Consider a ball $\hat{B}_0$ about $\hat{x}_0$ contained within $D$ of a radius $\hat{r}$ to be specified later; consider forthwith, the
`umbra' of $\hat{B}_0$ with respect to the origin i.e., the convex hull of $\hat{B}_0$ with the origin. Observe that this umbra is contained in $D$
as $D$ is star-like and contains a ball about $x_0$, say $B_0$ -- the radius of this ball can be taken to as large as 
$\hat{r}(\vert x_0 \vert/\vert \tilde{x}_0 \vert)$ as can be ascertained by a `similar-triangles' argument, but that is a digression. Getting to the point here, these balls are contained within $D$ to begin with.
\medskip\\
\noindent Pick any member $x_k$ of the sequence $\{x_j\}$, in the neighbourhood $B_0$ of $x_0$ and consider the ray $R_k$ shooting through $x_k$. As $B_0$
is contained in the aforementioned umbral region, $R_k$ must intersect $\hat{B}_0$ in a segment $s$, say. Notice that this segment must consist only of points whose distance from the origin is strictly bigger than $\vert \hat{x}_0 \vert -\hat{r}$. If this in turn can be shown to be bigger than
$\rho(x_k/\vert x_k \vert)$ for an appropriate choice
of $(\hat{x}_0,\hat{r})$, we will be done for this will imply that points of $s$ cannot belong to $D$ because $\rho(x_k/\vert x_ k \vert)$ is the upper 
threshold for a point on the ray $R_k$ to belong to $D$; this contradicts $\hat{B}_0 \subset D$.
\medskip\\
\noindent We shall choose the radius $\hat{r}$ small enough and 
the point $\hat{x}_0$, close enough to $\tilde{x}$ to the end of bounding $\vert \hat{x}_0 \vert -\hat{r}$ from below, as follows. 
Pick any positive number $M$ bigger than $\rho(x_0) - \epsilon$ and $\delta \in (0,1)$ with $1/\delta\; -1 <\epsilon/4M$. As
$\vert x_k \vert/\vert x_0 \vert \to 1$, we may choose $n \in \mathbb{N}$ large so that for all $k>n$,
\[
\Big\vert \frac{\vert x_k \vert}{\delta \vert x_0 \vert} - \frac{1}{\delta} \Big\vert < \frac{\epsilon}{4M}.
\]
Subsequently note that
\[
\Big\vert 1 - \frac{\vert x_k \vert}{\delta \vert x_0 \vert} \Big\vert < \Big\vert 1 - \frac{1}{\delta} \Big\vert
+ \Big\vert \frac{1}{\delta} - \frac{\vert x_k \vert}{\delta \vert x_0 \vert} \Big\vert < \frac{\epsilon}{2M}
\]
from which we get 
\begin{equation}\label{intercalc}
\rho(x_k) \frac{\big \vert\delta \vert x_0 \vert - \vert x_k \vert \big\vert}{\delta \vert x_0 \vert} < \frac{\epsilon}{2}.
\end{equation}
Now take for $\hat{r}$, any number less than $\epsilon \delta/2$. Then we have the following string of estimates:
\begin{align*}
\vert \hat{x}_0 \vert -\hat{r} &= \vert\delta \tilde{x}_0 \vert - \hat{r}\\
&> \delta \rho(x_0) \vert x_0 \vert - \epsilon \delta \vert x_0 \vert/2\\
&=\Big(\rho(x_0) - \epsilon/2 \Big) \delta \vert x_0 \vert \\
&>\Big(\rho(x_k) + \epsilon/2 \Big) \delta \vert x_0 \vert\\
&> \rho(x_k) \vert x_0 \vert + \rho(x_k) \Big\vert \delta \vert x_0 \vert - \vert x_k \vert \Big\vert\\
&> (1-\delta) \rho(x_k) \vert x_0 \vert + \rho(x_k) \vert x_k \vert \\
&>\rho(x_k / \vert x_k \vert).
\end{align*}
This leads to the sought-after contradiction as mentioned earlier and finishes the proof for the case when $\rho(x_0)$ is finite. The case when 
$\rho(x_0)$ is infinite is based on similar ideas with the details even simpler and therefore omitted.
\end{proof}

\begin{defn}
Let $B$ be ball about the origin and $p$ a point outside $B$. We shall refer to the convex hull of $B$ with $p$ as the umbra of $B$ with respect to $p$. 
Let $D$ be a domain in $\mathbb{R}^N$ containing the origin. We shall say that $D$ is umbral with respect to $p \in \partial D$ if there
is a ball $B_p$ about the origin contained in $D$, such that the {\it umbra of $B_p$ with vertex at $p$} i.e., the convex hull of $B_p$ with $p$, is contained 
in $D$. Then call $D$ an umbral domain.
\end{defn}
 
\noindent Observe that any such  domain $D$ is necessarily star-like with respect to the origin. We leave this as an easy exercise by formation of
successive umbras with vertex at the apriori possible various end-points of the intervals of intersection of a fixed ray $R$ with $D$ -- note that $R \cap D$ being 
a copy of an open subset of $\mathbb{R}$ must be expressible as a disjoint union of countably many open intervals, so that the set of end-points must be a discrete set.  

\begin{lem}
Let $D$ be any star domain not necessarily bounded, star-like with respect to the origin in $\mathbb{R}^N$. Let $R$ be a ray emanating from the origin
which hits $\partial D$. Suppose that $D$ is umbral with respect to every point in $R \cap \partial D$. Then $R$ actually intersects $\partial D$ 
at only one point. In particular if $D$ is umbral, any ray emanating from the origin, intersects the boundary $\partial D$ at most at one point. If 
$D$ is a bounded umbral domain then each such ray intersects the boundary at a unique point.
\end{lem}
\begin{proof}
By hypothesis, for any point $q \in R \cap \partial D$, there a ball $B_q$ about the origin contained in $D$ such that its convex hull with $q$ is also contained in $D$ except for $q$. If the ray $R$ emanating from the origin intersects $\partial D$ at more than one point, then the point $p$ of intersection nearest to
the origin on this ray is contained in the (interior of) the umbra of $B_q$ with vertex at any of the other points $q$ of $R \cap \partial D$, which in turn by
hypothesis is contained in $D\cup \{q\}$, contradicting the fact that $p$ is a boundary point of our domain $D$. 
\end{proof}

\noindent We remark that  among the star domains, say the bounded ones for simplicity, it is not the umbral ones alone which are proper i.e., have the unique boundary intersection property as in the foregoing theorem. Indeed, balls in the $l^p$-metric for $0<p<1$ i.e., $d(x,y) = \vert x -y \vert^p$, 
centered at the origin are star-like, non-umbral at the cusps but do have the property that each ray through the origin intersects the boundary at a unique point; in short are proper star domains. We remark that the terminology of umbras is not standard. They have been invented only to bring out certain
key features of complete Reinhardt domains, which leads to them being proper star domains, in particular.
Moreover infact, from the topological point of view, umbral domains are no more special than any star domain for, it is true that any star domain is homeomorphic
to the ball -- an excellent exposition can be found in \cite{Berg}. However, proper star domains which are bounded -- in particular, bounded umbral domains -- have the added feature that their closures are homeomorphic to the closed ball, of importance for us; for this reason and for immediate reference, 
we present a proof of this feature below. As our purpose here was to attain a topological characterization of complete Reinhardt 
domains and forthwith, study their boundaries and gauge function, we have chosen to include a presentation of their topological characterization (along-with their boundaries). To this end, we begin with: 

\begin{prop}
Every complete Reinhardt domain in $\mathbb{C}^N$ is umbral with respect to each of its boundary points away 
from the complex coordinate frame
$A$ and is consequently a proper star domain i.e., every ray intersects the boundary at most at one point.
\end{prop}

\begin{proof}
In view of the foregoing lemma, all what needs to be established is that a ray along say the first coordinate axis
 $\{ (t \zeta,0) \in \mathbb{C}^N \; : \; t \in\mathbb{R} \}$
where $\zeta \in \mathbb{C}$, intersects the boundary $\partial D$ at most at one point; likewise, for rays along any of the other coordinate axes.
To establish this, note that if 
\[
M_j := \sup\{ \vert z_j \vert \; : \; z \in D\} = \sup \{\vert z_j \vert \; : \; z \in \overline{D}\} \in (0, +\infty],
\]
then the intersection of $D$ with the $z_j$-axis is a disc $\Delta_j$ of radius $M_j$. If a ray along any of these axes, say the $j$-th one, intersects $\partial D$ at more than one point, then firstly, one -- infact, all but one -- of those points must have its $j$-th coordinate in modulus strictly less less than $M_j$. Let $r$
denote this set of intersection, with $p \in r$ being the point whose $j$-th coordinate in modulus is $M_j$ -- there is nothing to 
deal with $r$ is empty.
Let $q$ be any other point of $r$.
Both $p$ and $q$ must have all its coordinates $0$ except for the $j$-th one. As the intersection $\Delta_j$ is a disc and $q_j<M_j$, it
follows that the points of $D$ which occur in any neighbourhood of $p$ must have non-zero $j$-th coordinates. Picking any such point
and considering the polydisc spanned by it, we obtain points of $D$ on the $z_j$-axis, whose distance from the origin exceeds the radius 
of $\Delta_j$. Contradiction.
\end{proof}

\noindent We may now conclude that topologically any such domain is the same as the ball. Infact we may essentially obtain topological equivalence together with their boundaries. Further, the boundary of any bounded complete Reinhardt domain in $\mathbb{C}^N$ is homeomorphic to the sphere 
$\partial \mathbb{B}^N$. The study of the regularity of the boundaries of such domains is tantamount to that of their radial function, all dealt with 
in the following proposition.

\begin{prop}
Let $D$ be any star domain not necessarily bounded, star-like with respect to the origin in $\mathbb{R}^N$. Suppose that $D$ is proper i.e., each ray emanating from the origin intersects $\partial D$ at most at one point. The gauge functional and the radial function are continuous functions into the extended reals. Further when $D$ is bounded, the closure $\overline{D}$ is homeomorphic to $\overline{\mathbb{B}^N}$; the same is true when $D$ is unbounded provided certain points at infinity are appropriately appended to $\partial D$. 
\end{prop}
 
\begin{proof}
\noindent As the gauge and radial functions are reciprocals of each other, it suffices to show the continuity of $\rho = \rho_D$. Being radial, it suffices to further restrict attention to showing continuity of $\rho$ on the unit sphere. Recall that for a unit
vector $v$, $\rho(v)$ is the distance of the origin to $\partial D$ along the ray spanned by $v$ namely, $\{ \lambda v \; :\; \lambda\geq 0\}$. Note that there is no reference to any norm in the definition of $\rho$; the distance last mentioned is with respect to that norm with respect to which $v$ is a unit vector -- the norm as usual is the $l^2$-norm unless otherwise stated. Thus our task boils down to verifying continuity of the variation of the distance 
from the origin of  points on the boundary, essentially. To be rigorous, first note that the mapping $\partial D \to \partial \mathbb{B}^N$ given by
$ x \to x/\vert x \vert$ is obviously continuous, is injective by the hypothesis on $\partial D$ and is surjective owing to the  boundedness of $D$. Altogether we have a one-to-one, continuous map from the compact set $\partial D$ onto $\partial \mathbb{B}^N$; such a mapping must be a homeomorphism.
Therefore, its inverse $x \to \rho(x)x$ must be continuous from which it is easy to argue the continuity of $\rho$. While we can extend these maps to furnish a homeomorphism between $D$ and $\mathbb{B}^N$, we shall proceed forthwith to the case when $D$ is unbounded, the dealing of which will take care of this case as well.
\medskip\\
\noindent Having dropped the assumption of boundedness on $D$, we intend to use the following variant of the key tool used in the foregoing proof of continuity of $\rho$: a proper, injective continuous mapping must be a homeomorphism. Recall that a map is termed proper, if every compact subset of the 
range has its inverse image compactly contained in the domain of the map. Equivalently, a mapping $f: D_1 \to D_2$ between a pair of domains 
in $\mathbb{R}^N$ is {\it proper} if for every sequence $\{x_j\}$ which accumulates only on $\partial D_1 \cup \{\infty\}$ -- with $\infty$ to be understood as the singleton required to be appended to $\mathbb{R}^N$ in its one-point compactification and being relevant here only when the domain is unbounded -- 
the image sequence $\{f(x_j)\}$ also has the property that it clusters only on $\partial D_2 \cup \{ \infty\}$; this is expressed in short by saying that 
$f$ preserves `boundary sequences'.  Indeed, to now furnish a rigorous proof in the case when $D$ is not necessarily bounded, first let $E:= \{ x \in D \setminus \{0\} \; :\; \rho(x) = +\infty\}$. Consider again the map $N: x \to x /\vert x \vert$, this time as a mapping
\[
N: \partial D \to \partial \mathbb{B}^N \setminus N(E).
\]
Then again $N$ is injective by virtue of the hypothesis that each ray emanating from the origin hits $\partial D$ at most once (at a point where
$\rho$ is finite). Note that $N(E)$ is precisely the set of points on $\partial \mathbb{B}^N$ at which $\rho$
is infinite, which also enables the observation that
if $\{x_k\}$ is a sequence in the closed set $\partial D$ which is `compactly divergent' or more simply put 
$\vert x_k \vert \to \infty$, then along the sequence $v_k := x_k/ \vert x_k \vert$ on the unit sphere,
\begin{equation} \label{rhofact}
\rho(v_k) = \vert x_k \vert \rho(x_k) = \vert x_k \vert  \to \infty.
\end{equation}
Here we have relied on the unique intersection property of the hypotheses to conclude for our $x_k \in \partial D$ that $\rho(x_k)=1$. After passing to a subsequence if necessary, we may assume that 
$\{v_k\}$ is convergent to $v_0 \in \partial \mathbb{B}^N$ with $\rho(v_0)$ is finite. Next, let $\hat{v}_0 := \rho(v_0)v_0$ and $\tilde{v}_0$ any 
point on the ray $R_0$ through $v_0$ with $\vert \tilde{v}_0 \vert > \vert \hat{v}_0 \vert$. Then first of all note that $\tilde{v}_0$ must
lie in the domain $\tilde{D}= \mathbb{C}^N \setminus \overline{D}$. Thereafter consider a ball $\tilde{B}$ inside $\tilde{D}$ about $\tilde{v}_0$. 
Observe that the ray through any point $v_k$ sufficiently close to $v_0$ will enter $\tilde{B} \subset \tilde{D}$: choose
$k$ large such that ${\rm dist}(v_0,v_k)< \delta(\vert v_0 \vert/\vert \tilde{v}_0 \vert)$ where $\delta$ is the radius of $\tilde{B}$; a `similar triangles' argument shows that the ray through 
$v_k$ indeed enters $\tilde{B}$ which lies away from $\overline{D}$. By definition of $\rho_D(v_k)$, this forces $\rho(v_k)$ to be bounded above by the number $\rho(v_0) + \delta$ to contradict the fact that $\rho(v_k) \to \infty$. We conclude therefore that $\rho(v_0)$ must be infinite and
$v_0 \in N(E)$. This means that $N(v_k)$ accumulates only on $N(E)$ showing that the mapping $N$ as above, is a proper, one-to-one,
continuous map from $\partial D$ onto $\partial \mathbb{B}^N \setminus N(E)$, hence a homeomorphism. Thus, its inverse, which is just the map 
$x \to \rho(x)x$ must be continuous, which implies the continuity of $\rho$ on $\partial \mathbb{B}^N \setminus N(E)$. This assures that the following map
is continuous:
\[
\Phi(x) = \Big(\frac{\vert x \vert}{1+ \frac{1}{\rho(x/\vert x \vert)} - \vert x \vert} \Big) x,
\]
as a mapping from $\mathbb{B}^N$ onto $D$, extended to the origin by setting $\Phi(0)=0$ and with the understanding that the reciprocal of 
$\rho(x/ \vert x \vert)$ is to be replaced by zero, whenever $x \in \mathbb{B}^N \setminus \{0\}$ has $\rho(x)$ infinite. We note that this map is a
one-to-one, continuous mapping with the following boundary behaviour: when $x \to v \in \partial \mathbb{B}^N$ with $\rho(v)$ finite, 
then $\Phi(x) \to \rho(v)v \in \partial D$; whereas, for any sequence $x_n \in \mathbb{B}^N$ with $x_n \to v$ whose $\rho(v) = \infty$, we have
$\vert \Phi(x_n) \vert \to \infty$, so  $\Phi(x_n)$ cannot accumulate anywhere in $D$. This means that $\Phi$ is a one-to-one, proper continuous mapping 
allowing us to conclude that $\Phi$ must be a homeomorphism, as desired. We conclude with the observation that $\partial D$ is homeomorphic to
a spherically convex subset of $\partial \mathbb{B}^N$.
\end{proof}


\begin{thebibliography}{BFKKMP}

\bibitem{BS} Bracci, Filippo; Saracco, Alberto:
\textit{Hyperbolicity in unbounded convex domains}, 
Forum Math. 21 (2009), no. 5, 815–-825. 

\bibitem{Berg} Berger, Marcel:
\textit{Geometry. I.}
Translated from the French by M. Cole and S. Levy. Universitext. Springer-Verlag, Berlin, 1987. xiv+428 pp.

\bibitem{Bo} Boas, Harold:
\textit{Math 650: Several Complex Variables, Fall-2013.}, http://www.math.tamu.edu/~boas/courses/650-2013c/

\bibitem{Bo2} Boas, Harold:
\textit{Majorant Series}, J. Korean Math. Soc. 37 (2000), No. 2, pp. 321--337.

\bibitem{dAn} D'Angelo, John P.:
\textit{Several complex variables and the geometry of real hypersurfaces}, Studies in Advanced Mathematics. CRC Press, Boca Raton, FL, 1993.

\bibitem{Ghomi} Ghomi, Mohammad:
\textit{Deformations of unbounded convex bodies and hypersurfaces}, Amer. J. Math. 134 (2012), no. 6, 1585–-1611.

\bibitem{R1} Range, R. Michael:
\textit{What is $\ldots$ a pseudoconvex domain?}, Notices Amer. Math. Soc. 59 (2012), no. 2, 301–-303.

\bibitem{R2} Range, R. Michael:
\textit{Complex analysis: a brief tour into higher dimensions}, Amer. Math. Monthly 110 (2003), no. 2, 89–-108.

\bibitem{S} Shabat, B. V.:
\textit{Introduction to complex analysis. Part II. Functions of several variables.} Translated from the third (1985) Russian edition by J. S. Joel. Translations of Mathematical Monographs, 110. American Mathematical Society, Providence, RI, 1992. x+371 pp.

\bibitem{Sch} Schneider, Rolf:
\textit{Convex bodies: the Brunn-Minkowski theory}, Second expanded edition. Encyclopedia of Mathematics and its Applications, 151. Cambridge University Press, Cambridge, 2014. xxii+736 pp. 

\bibitem{JarPflu} Jarnicki, Marek; Pflug, Peter:
\textit{First steps in several complex variables: Reinhardt domains},
EMS Textbooks in Mathematics. European Mathematical Society (EMS), Zürich, 2008. viii+359 pp. 

\bibitem{JarPflu2} Jarnicki, Marek; Pflug, Peter:
\textit{Invariant distances and metrics in complex analysis}, Second extended edition. de Gruyter Expositions in Mathematics, 9. Walter de Gruyter, 2013. xviii+861 pp.

\bibitem{H} H\"ormander, Lars:
\textit{Notions of convexity}, Modern Birkh\"auser Classics. Birkhäuser Boston, Inc., Boston, MA, 2007.


\bibitem{IKra} Isaev, Alexander V.; Krantz, Steven G.:
\textit{Invariant distances and metrics in complex analysis},
Notices Amer. Math. Soc. 47 (2000), no. 5, 546-–553. 


\bibitem{L} Lempert, L.:
\textit{La m\'e trique de Kobayashi et la repr\'e sentation des domaines sur la boule.} Bull. Soc.
Math. Fr. \textbf{109}, 427–-474 (1981).

\bibitem{R} Range, R. Michael:
\textit{Holomorphic functions and integral representations in several complex variables},
Graduate Texts in Mathematics, 108. Springer-Verlag, New York, 1986. xx+386 pp. 


\bibitem{HM} Herbig, A.-K.; McNeal, J. D:
\textit{Convex defining functions for convex domains}, J. Geom. Anal. 22 (2012), no. 2, 433–-454. 

\bibitem{Oh} Ohsawa, Takeo:
\textit{Analysis of several complex variables}, Translated from the Japanese by Shu Gilbert Nakamura. Translations of Mathematical Monographs, 211. Iwanami Series in Modern Mathematics. American Mathematical Society, Providence, RI, 2002. xviii+121 pp. 

\bibitem{MT} Magaril-Ilʹyaev, G. G.; Tikhomirov, V. M.:
\textit{Convex analysis: theory and applications}, Translated from the 2000 Russian edition by Dmitry Chibisov and revised by the authors. Translations of Mathematical Monographs, 222. American Mathematical Society, Providence, RI, 2003. viii+183 pp.

\bibitem{St} Stoker, J. J.:
\textit{Unbounded convex point sets},
Amer. J. Math. 62, (1940). 165–-179. 

\bibitem{V} Vladimirov, Vasiliy Sergeyevich:
\textit{Methods of the theory of functions of many complex variables},
Translated from the Russian by Scripta Technica, Inc. Translation edited by Leon Ehrenpreis The M.I.T. Press, Cambridge, Mass.-London 1966 xii+353 pp. 

\end{thebibliography}
\end{document}